\documentclass{amsart}

\usepackage{amssymb,color}
\usepackage{amsfonts}
\usepackage{amsmath}
\usepackage{euscript}
\usepackage{enumerate}
\usepackage{pdfsync}
\newtheorem{theorem}{Theorem}[section]
\newtheorem{lemma}[theorem]{Lemma}
\newtheorem{note}[theorem]{Note}
\newtheorem{prop}[theorem]{Proposition}
\newtheorem{cor}[theorem]{Corollary}

\newtheorem*{Theorem1'}{Theorem 1'}

\theoremstyle{definition}

\theoremstyle{remark}

\newcommand \Z{{\mathbb Z}}
\newcommand \Q{{\mathbb Q}}
\newcommand \N{{\mathbb N}}

\begin{document}

\title[Structure of the Macdonald groups in one parameter]{Structure of the Macdonald groups in one parameter}

\author{Alexander Montoya Ocampo}
\address{Department of Mathematics and Statistics, University of Regina, Canada}
\email{alexandermontoya1996@gmail.com}

\author{Fernando Szechtman}
\address{Department of Mathematics and Statistics, University of Regina, Canada}
\email{fernando.szechtman@gmail.com}
\thanks{The second author was partially supported by NSERC grant RGPIN-2020-04062}





\begin{abstract} Consider the Macdonald groups $G(\alpha)=\langle A,B\,|\, A^{[A,B]}=A^\alpha,\, B^{[B,A]}=B^\alpha\rangle$, $\alpha\in\Z$.
We fill a gap in Macdonald's proof that $G(\alpha)$ is always nilpotent, and proceed to determine the order, upper and lower central series,
nilpotency class, and exponent of $G(\alpha)$.
\end{abstract}

\maketitle

\section{Introduction}

In 1962, Macdonald \cite{M} investigated the group
$$
G(\beta,\gamma)=\langle A,B\,|\, A^{[A,B]}=A^\beta,\, B^{[B,A]}=B^\gamma\rangle,\quad \beta,\gamma\in\Z.
$$

In this paper, our main object of study is the Macdonald group in one parameter $\alpha\in\Z$, namely the group $G=G(\alpha)=G(\alpha,\alpha)$.
If $\alpha=1$ then $G$ is infinite and isomorphic to the integral Heisenberg group \cite[Chapter 5]{J}.  If $\alpha\neq 1$, then
\cite[Section 4]{M} shows that $|G|$ is finite with the same prime factors as $\alpha-1$. 
In particular, $G(0)$ and $G(2)$ are trivial. For these reasons, we assume for the remainder of the paper that $\alpha\notin\{0,1,2\}$.

Our motivation for this study is manifold. First of all, Macdonald's paper \cite{M} is beautifully written,
and he determined various basic structural properties of the groups
$G(\beta,\gamma)$ through elegant arguments. However, his proof of the nilpotence of $G(\beta,\gamma)$ is incomplete. 
He did show in detail that $G(\beta,\gamma)$ is nilpotent when $\beta=1$, or $\gamma=1$, or $\gcd(\beta-1,6)=1=\gcd(\gamma-1,6)$,
and stated that the remaining cases followed similarly. We beg to differ. Our discrepancy stems from an innocent
looking error, found in \cite[Section 5]{M}, where
it is stated that if $\beta>1$, $\beta=1+3k$, and $k\equiv -1\mod 3$, then for
$\delta_\beta=\beta^\beta-(1+\beta+\cdots+\beta^{\beta-1})$,
one has $v_3((\beta-1)\delta_\beta)=4$. This is only valid if $v_3(k+1)=1$, as seen in Proposition \ref{calc1} below.
If $\beta>1$ and $\gamma>1$, the values of 
$v_p((\beta-1)\delta_\beta)$ and $v_q((\gamma-1)\delta_\gamma)$, 
where $p$ and $q$ are prime factors of $\beta-1$ and $\gamma-1$, respectively,
play an essential role in the structure of $G(\beta,\gamma)$ and in particular in bounding the order and nilpotency class of 
$G(\beta,\gamma)$ (provided nilpotence is first established). As seen in Sections~\ref{secv}~and~\ref{secn},
considerable effort is required to establish the nilpotence of $G(\alpha)$, in Theorem \ref{nul}, especially when $\alpha\equiv 7\mod 9$.

Secondly, Macdonald left open the question of the precise order and class of his groups $G(\beta,\gamma)$
as complicated. 
He came back to this question in \cite{M2}, ten years after the appearance of \cite{M}, with a general computer program that
allowed him, in particular, to determine the order and nilpotency class of the Sylow 3-subgroup of $G(34,7)$, found to be 
$3^{10}$ and 7, respectively,
less than the bounds given in \cite{M}. In this paper, we settle this problem for the groups $G(\alpha)$ in
Theorem \ref{todojunto}.

Thirdly, for $\beta$ and $\gamma$ different from one, $G(\beta,\gamma)$ is a finite group that
admits a finite presentation with as many generators as relations. These groups are called
{\em interesting} by Johnson \cite[Chapter 7]{J}. A more widely used synonym is finite group
with deficiency zero. Such groups have been studied by several authors. The intensity of
this study increased since the discovery by Mennicke \cite{Me} 
of the first family of finite groups requiring 3 generators and 3 relations, in 1959.
He considered the groups $M(a,b,c)=\langle x,y,z\,|\, x^y=x^a, y^{z}=y^b, z^x=z^c\rangle$ and  proved them
to be finite when $a=b=c\geq 2$. Note the similarity between the presentations of the trivial group $M(2,2,2)$ and 
Higman's infinite group \cite{H} from 1951. It is easy to see that $M(a,b,c)$ does require 3 generators whenever $a-1,b-1,c-1$ share a prime factor. The problem of the finiteness of the general Mennicke groups
$M(a,b,c)$ was studied by Macdonald and Wamsley \cite{W}, Schenkman~\cite{S}, and Jabara~\cite{Ja}. A sufficient condition
is that $|a|,|b|,|c|\geq 2$. Upper bounds for the order of 
 $M(a,b,c)$ were provided by Johnson and Robertson \cite{JR}, Albar and Al-Shuaibi \cite{AA}, and Jabara \cite{Ja}.
The actual order of $M(a,b,c)$ is known only in certain cases (see \cite{Me,A,AA,Ja}). As exemplified
by the Mennicke groups $M(a,b,c)$, it may be quite difficult to determine the order and other structural
properties of the members of a given family of finite groups of deficiency zero.

Fourthly, in 1970 the groups 
$
W_{\pm}(\beta,\gamma,\delta)=\langle X,Y,Z\,|\, X^Z=X^\beta, Y^{Z^{\pm 1}}=Y^\gamma, Z^\delta=[X,Y]\rangle
$
were shown to be finite by Wamsley \cite {W2}, provided $\beta^{|\delta|}\neq 1,\gamma^{|\delta|}\neq 1$, $\delta\neq 0$. Very little
is known about these interesting groups. Our study of $G(\alpha)$ plays a crucial role in \cite{PS} in order to describe the Wamsley groups
when $\beta=\gamma$.

In the fifth place, once the nilpotence of the finite group $G(\alpha)$ is established, it follows that
$G(\alpha)$ is the direct product of its Sylow subgroups $G(\alpha)_p$, where $p\in\N$ runs through
all prime factors of $\alpha-1$, and our study of $G(\alpha)$ yields detailed 
structural information about the infinite family of finite $p$-groups $G(\alpha)_p$.
It should be noted that, according to Wamsley \cite{W3}, the Sylow subgroups $G(\alpha)_p$ are themselves interesting,
except perhaps when $p=2$.
The study of finite $p$-groups, is quite active, as seen in \cite{AT, B, BJ, BJ2, BJ3, BJ4, G, IY, LM, TW}, for instance.
We hope that our investigation of the groups $G(\alpha)_p$ contributes in this regard. We find a presentation 
as well as the order, upper and lower central series, nilpotency class, and exponent of $G(\alpha)_p$, among other structural
properties. See Theorems \ref{predela}, \ref{main1}, and \ref{main2}.

Lastly, the automorphism groups of finite $p$-groups, have been the object of considerable attention,
from various angles, one of which is simply to investigate the actual structure of $\mathrm{Aut}(P)$ when $P$ is a finite 
$p$-group of one type or another, see \cite{BC, C, C2, D, Di, GG, Ma, Ma2, Men}, for instance. As an application
of the present paper, we elucidate in \cite{MS,MS2} the structure of $\mathrm{Aut}(G(\alpha)_p)$, except when $p=3$
and $\alpha\equiv 7\mod 9$. This, in turn, allows us in \cite{MS3} to find necessary and sufficient conditions for $G(\alpha)_p$
to be isomorphic to $G(\beta)_p$, a decidedly nontrivial problem involving finite $p$-groups.

The paper is organized as follows. Sections \ref{secv} and \ref{secn} contain the required calculations
for a proof that $G(\alpha)$ is nilpotent.
In Sections \ref{secv2}, \ref{secv3}, and~\ref{secv4} we give a presentation
for each Sylow subgroup $G(\alpha)_p$ of $G(\alpha)$ and find an upper bound for its order. The proof that this
upper bound is sharp is fairly laborious and is postponed to an appendix in order to maintain the flow of the paper (of course,
this appendix does not depend on any intermediate results). Relying on the constructions given in the appendix, namely
Theorems \ref{model}, \ref{model2}, and \ref{model3},
Section \ref{secv6} yields the order of $G(\alpha)_p$, a normal form for its elements, as well as some basic structural properties.
The upper and lower central series of $G(\alpha)_p$, together with its nilpotency class, are found in Section \ref{secv7}.
The exponent of $G(\alpha)_p$ is determined in Section \ref{secv8}. All of this is combined in Section \ref{secv9} to
produce the order, nilpotency class, and exponent of $G(\alpha)$ in terms of~$\alpha$.

In terms of notation, function composition proceeds from left to right. Given a group $T$, we~set
$$
[a,b]=a^{-1}b^{-1}ab,\; b^a=a^{-1}ba,\; {}^a b=aba^{-1},\quad a,b\in T.
$$
Moreover, we let $\langle 1\rangle=Z_0(T),Z_1(T),Z_2(T),\dots$ stand for the terms of the upper central series of~$T$, so that 
$Z_{i+1}(T)/Z_i(T)$ is the center of $T/Z_i(T)$, and write $T=\gamma_1(T),\gamma_2(T),\gamma_3(T),\dots$ for the terms of the lower central series of~$T$, so that 
$\gamma_{i+1}(T)=[T,\gamma_i(T)]$. Furthermore, for $n\in\Z$, we let $T^n$ stand for the subgroup of $T$ generated by
all $t^n$, with $t\in T$. Finally, given a ring $R$ with $1\neq 0$, we write $\mathrm{Heis}(R)$ for the Heisenberg group over $R$
(see \cite[Chapter 5]{J} for the case $R=\Z$).

Recall our general assumption that $\alpha$ is an integer satisfying $\alpha\notin\{0,1,2\}$.

\section{Valuation calculations}\label{secv}

Consider the following function of $\alpha$:
\begin{equation}
\label{gamita}
\gamma=\alpha^\alpha-(1+\alpha+\cdots+\alpha^{\alpha-1}),\text{ if }\alpha>0.
\end{equation}

\begin{prop}\label{calc1} Suppose $\alpha>0$, $p\in\N$ is a prime factor of $\alpha-1$, and $v_p(\alpha-1)=m$. Then
$$
v_p(\gamma)=\begin{cases} 2m & \text{ if }p>3,\\ 2m-1 & \text{ if }p=2,\\
2m & \text{ if }p=3,\text{ with }m>1\text{ or }(\alpha-1)/3\equiv 1\mod 3,\\
2+s & \text{ if }p=3,\text{ where }\alpha=1+3k,\;k=-1+3^s u,\; s,u\in\N,\text{ and }\gcd(3,u)=1.
\end{cases}
$$
Moreover, in the last case, we have  $\gamma=3^{2+s}t$, where $t\in\N$ and $t\equiv -u\mod 3$.
\end{prop}

\begin{proof} It follows from
\cite[pp. 611]{M} that
$$(\alpha-1)\gamma=\frac{[(\alpha-1)^2+1]}{2}(\alpha-1)^3+(\alpha-2){{\alpha}\choose{3}}(\alpha-1)^3+
(\alpha-2){{\alpha}\choose{4}}(\alpha-1)^4+\cdots,$$ so $v_p(\gamma)=2m$ if $p>3$, and $v_p(\gamma)=2m-1$ if $p=2$. 
If $p=3$ put $\alpha=1+3^m k$ with $3\nmid k$, so that $$(\alpha-1)\gamma=
\frac{[1+3^{2m}k^2+(-1+3^m k)^2(1+3^m k)3^{m-1}k]}{2}(\alpha-1)^3+(\alpha-2){{\alpha}\choose{4}}(\alpha-1)^4+\cdots.$$
If $m>1$ or $k\equiv 1\mod 3$, we infer $v_3(\gamma)=2m$. Suppose next $p=3$, $m=1$ and $k=-1+3^s u$, where $s,u\in\N$ and
$\gcd(3,u)=1$. We then have
$$
\alpha^\alpha
= \alpha^{\alpha-1}\alpha
= \alpha^{\alpha-1}(1+(\alpha-1))
= \alpha^{\alpha-1} + (\alpha-1)\alpha^{\alpha-1},
$$
so
$$
    \gamma
    = \alpha^\alpha - (1 + \alpha + \cdots + \alpha^{\alpha - 1})
    = \alpha^\alpha - \alpha^{\alpha - 1} - (1 + \alpha + \cdots + \alpha^{\alpha - 2})\\
    = (\alpha-1)\alpha^{\alpha-1} - \frac{\alpha^{\alpha-1}-1}{\alpha-1},
$$		
that is,
\begin{equation}
\label{buce}
\gamma=3k(1+3k)^{3k}-\frac{(1+3k)^{3k}-1}{3k}.
\end{equation}
Now
$$
(1+3k)^{3}=1+3^2k+3^3k^2+3^3k^3=1+3^2k+3^3k^2(1+k)=1+3^2k+3^{3+s}k^2u,
$$
and therefore
$$
(1+3k)^{3k}=(1+3^2k+3^{3+s}k^2u)^k=\underset{i+j+\ell=k}\sum\frac{k!}{i!j!\ell !} 1^i 3^{2j}k^{j} 3^{(3+s)\ell} (k^2u)^\ell.
$$
Splitting off the case $\ell=0$, we obtain
$$
(1+3k)^{3k}=\underset{0\leq j\leq k}\sum {{k}\choose{j}} 3^{2j}k^{j}+
\underset{\ell\geq 1, i+j+\ell=k}\sum\frac{k!}{i!j!\ell  !} 1^i 3^{2j}k^{j} 3^{(3+s)\ell} (k^2u)^\ell.
$$
All terms in the second summand, except for the term with $j=0,\ell=1$, are multiples of $3^{3+s+1}k$, so
$$
(1+3k)^{3k}=\left(\underset{0\leq j\leq k}\sum {{k}\choose{j}} 3^{2j}k^{j}\right)+ 3^{3+s}k^3u+3^{3+s+1}kw,\quad w\in\Z,
$$
$$
3k(1+3k)^{3k}=\left(\underset{0\leq j\leq k}\sum {{k}\choose{j}} 3^{2j+1}k^{j+1}\right)+ 3^{3+s}r,\quad r\in\Z,
$$
$$
\frac{(1+3k)^{3k}-1}{3k}=\left(\underset{1\leq j\leq k}\sum {{k}\choose{j}} 3^{2j-1}k^{j-1}\right)+ 3^{2+s}k^2u+3^{3+s}w.
$$
Since $\gcd(3,k^2u)=1$, we may now go back to (\ref{buce}) and deduce that,
provided $3^{3+s}$ is a factor of
$$
P=\underset{0\leq j\leq k}\sum {{k}\choose{j}} 3^{2j+1}k^{j+1}-\underset{1\leq j\leq k}\sum {{k}\choose{j}} 3^{2j-1}k^{j-1},
$$
we indeed have $v_3(\gamma)=2+s$, with $\gamma=3^{2+s}t$, $t\in\N$, $t\equiv -k^2u\equiv -u\mod 3$. To see that $3^{3+s}\mid P$,
put $f=j-1$, so that $j=f+1$, $2j-1=2f+1$, and
$$
\underset{1\leq j\leq k}\sum {{k}\choose{j}} 3^{2j-1}k^{j-1}=
\underset{0\leq f\leq k-1}\sum {{k}\choose{f+1}} 3^{2f+1}k^{f}.
$$
Therefore
\begin{align*}
P &=\underset{0\leq f\leq k}\sum {{k}\choose{f}} 3^{2f+1}k^{f+1}-\underset{0\leq f\leq k-1}\sum {{k}\choose{f+1}} 3^{2f+1}k^{f}\\
 &=3^{2k+1}k^{k+1}+\underset{0\leq f\leq k-1}\sum 3^{2f+1}k^{f}\left(k {{k}\choose{f}}-{{k}\choose{f+1}}\right)\\
&=3^{2k+1}k^{k+1}+\underset{0\leq f\leq k-1}\sum 3^{2f+1}k^{f}{{k}\choose{f}}\frac{f(k+1)}{f+1}\\
&=3^{2k+1}k^{k+1}+\underset{1\leq f\leq k-1}\sum 3^{2f+1}k^{f}{{k}\choose{f}}\frac{f3^s u}{f+1}.
\end{align*}
As $3^su=k+1$, we have $k+1\geq 3^s$. On the other hand, $3^s=(1+2)^s\geq 1+2s>1+s$, so $2(k-1)\geq k-1\geq s$,
and therefore $2k+1\geq 3+s$. It follows that $v_3(3^{2k+1}k^{k+1})=2k+1\geq 3+s$. Fix any $f$ such that $1\leq f\leq k-1$,
and set $g=v_3(f+1)$. Then $f+1\geq 3^g$, which implies, as above, that $f-1\geq g$. Thus, $\frac{3^s u}{f+1}\in\Q$, with
$v_3(\frac{3^s u}{f+1})=s-g\geq s-(f-1)$. Moreover, ${{k}\choose{f}}f\in\N$, so 
$v_3({{k}\choose{f}}f)\geq 0$. Our calculation of $P$ shows that 
$3^{2f+1}k^{f}{{k}\choose{f}}\frac{f3^s u}{f+1}$
is a positive integer, and we have
$$
v_3\left( 3^{2f+1}k^{f}{{k}\choose{f}}\frac{f3^s u}{f+1}\right)\geq 2f+1+(s-g)\geq 2f+1+s-(f-1)=s+f+2\geq s+3.
$$
\end{proof}

Consider the following function of $\alpha$:
\begin{equation}
\label{gamita2}
\mu=\alpha^{\alpha^2+1}-(1+\alpha+\cdots+\alpha^{\alpha^2-1}).
\end{equation}

\begin{prop}\label{calc2} Suppose that $\alpha\neq -1$, and let $p\in\N$ be a prime factor of $\alpha-1$ with $v_p(\alpha-1)=m$. Then 
$\mu\neq 0$, $v_p((\alpha-1)\mu)\geq 3m$ if $p=2$, and 
$$
v_p((\alpha-1)\mu)=\begin{cases} 3m & \text{ if }p>3,\\
4 & \text{ if }p=3,\alpha=1+3k, k\in\Z, k\equiv -1\mod 9.
\end{cases}
$$

\end{prop}

\begin{proof} As $|\alpha|\geq 2$, we see that $|\alpha^{\alpha^2+1}|>|1+\alpha+\cdots+\alpha^{\alpha^2-1}|$,
so $\mu\neq 0$.

We have $(\alpha-1)\mu=\alpha^{\alpha^2+2}-\alpha^{\alpha^2+1}-(\alpha^{\alpha^2}-1)=
\alpha^{\alpha^2}(\alpha^2-\alpha-1)+1$, where
$$
\alpha^{\alpha^2}=(1+(\alpha-1))^{\alpha^2}=1+{{\alpha^2}\choose{1}}(\alpha-1)+{{\alpha^2}\choose{2}}(\alpha-1)^2+
{{\alpha^2}\choose{3}}(\alpha-1)^3+{{\alpha^2}\choose{4}}(\alpha-1)^4+\cdots.
$$
Thus, setting $\nu=(\alpha^2-\alpha-1)$, we deduce
\begin{align*}
(\alpha-1)\mu & =\alpha(\alpha-1)+\nu{{\alpha^2}\choose{1}}(\alpha-1)+\nu{{\alpha^2}\choose{2}}(\alpha-1)^2+
\nu{{\alpha^2}\choose{3}}(\alpha-1)^3+\nu{{\alpha^2}\choose{4}}(\alpha-1)^4+\cdots\\ &=
\frac{((\alpha-1)^2+3(\alpha-1)+2)}{2}((\alpha-1)^2(\alpha+1)+1)(\alpha-1)^3+\nu{{\alpha^2}\choose{3}}(\alpha-1)^3+\cdots,
\end{align*}
which proves the result when $p\neq 3$. Suppose next $p=3$ and $\alpha=1+3k$, where $k\in\Z$ satisfies $k\equiv -1\mod 9$.
Put
$$
S=\alpha(\alpha+1),\; T=S\frac{\alpha\nu}{2},
$$
$$
U=T\frac{(\alpha-1)(\alpha^2-2)}{3},\; V=U\frac{(\alpha^2-3)(\alpha-1)}{4}.
$$
Then
$$
(\alpha-1)\mu=(S+T+U+V)(\alpha-1)^3+\nu{{\alpha^2}\choose{5}}(\alpha-1)^5+\nu{{\alpha^2}\choose{6}}(\alpha-1)^6+\cdots,
$$
where 
$$
S\equiv 2,\; T\equiv -1,\;U\equiv 2,\; V\equiv 3\quad\mod 9.
$$
Thus $S+T+U+V\equiv -3\mod 9$, so $v_3(S+T+U+V)=1$, whence $v_3((\alpha-1)\mu)=4$.
\end{proof}

\section{Nilpotence of $G$}\label{secn}

We maintain the following presentation of $G=G(\alpha)$ throughout this section:
$$
\langle A,B\,|\, A^{[A,B]}=A^\alpha,\, B^{[B,A]}=B^\alpha\rangle,
$$
setting $C=[A,B]$, and observing the existence of an automorphism $A\leftrightarrow B$, say $\theta$ of $G$, satisfying 
$C\leftrightarrow C^{-1}$. Recalling the definition (\ref{gamita2}) of $\mu$, we set $\mu_0=\alpha\mu$.

\begin{prop}\label{calc3} We have
$$
A^{\mu_0}B^{\mu_0}=1,\; A^{(\alpha-1)\mu_0}=1=B^{(\alpha-1)\mu_0}.
$$
\end{prop}

\begin{proof} If $\alpha=-1$ then $\mu_0=0$ and there is nothing to do.
Suppose henceforth that $\alpha\neq -1$. As conjugation by $C^2$ is an automorphism of $G$, we see that
\begin{equation}\label{iq} [A,B^{\alpha^2}]^{C^2}=[A^{\alpha^2},B].
\end{equation}

Regarding the left hand side of (\ref{iq}), we have
$$
(B^{\alpha^2})^A=(B^A)^{\alpha^2}=(BC^{-1})^{\alpha^2}=C^{-\alpha^2}B^{\alpha(1+\alpha+\cdots+\alpha^{\alpha^2-1})},
$$
which successively implies
$$
(B^{-\alpha^2})^A=B^{-\alpha(1+\alpha+\cdots+\alpha^{\alpha^2-1})}C^{\alpha^2},
$$
\begin{equation}\label{iq2}
[A,B^{\alpha^2}]=(B^{-\alpha^2})^A B^{\alpha^2}=B^{-\alpha(1+\alpha+\cdots+\alpha^{\alpha^2-1})}C^{\alpha^2}B^{\alpha^2},
\end{equation}
\begin{equation}\label{iq3}
[A,B^{\alpha^2}]^{C^2}=C^{-2}B^{-\alpha(1+\alpha+\cdots+\alpha^{\alpha^2-1})}C^{\alpha^2}B^{\alpha^2}C^2.
\end{equation}

As for the right hand side of (\ref{iq}), applying $\theta$ to (\ref{iq2}) and then inverting yields
\begin{equation}\label{iq4}
[A^{\alpha^2},B]=A^{-\alpha^2}C^{\alpha^2}A^{\alpha(1+\alpha+\cdots+\alpha+\alpha^{\alpha^2-1})}.
\end{equation}
It follows from (\ref{iq}) that the right hand sides of (\ref{iq3}) and (\ref{iq4}) are equal. Thus
\begin{align*}
B^{-\alpha(1+\alpha+\cdots+\alpha^{\alpha^2-1})}C^{\alpha^2}B^{\alpha^2}C^2 &=
C^2A^{-\alpha^2}C^{\alpha^2}A^{\alpha(1+\alpha+\cdots+\alpha^{\alpha^2-1})}\\ &=
C^2A^{-\alpha^2}C^{-2}C^{\alpha^2+2}A^{\alpha(1+\alpha+\cdots+\alpha^{\alpha^2-1})}\\ &=
A^{-1}C^{\alpha^2+2}A^{\alpha(1+\alpha+\cdots+\alpha^{\alpha^2-1})}.
\end{align*}
On the other hand,
\begin{align*}
B^{-\alpha(1+\alpha+\cdots+\alpha^{\alpha^2-1})}C^{\alpha^2}B^{\alpha^2}C^2 &=
B^{-\alpha(1+\alpha+\cdots+\alpha^{\alpha^2-1})}C^{\alpha^2+2}C^{-2}B^{\alpha^2}C^2\\ &=
B^{-\alpha(1+\alpha+\cdots+\alpha^{\alpha^2-1})}C^{\alpha^2+2}B\\ &=
B^{-\alpha(1+\alpha+\cdots+\alpha^{\alpha^2-1})}C^{\alpha^2+2}B C^{-(\alpha^2+2)}C^{\alpha^2+2}\\ &=
B^{-\alpha(1+\alpha+\cdots+\alpha^{\alpha^2-1})}B^{\alpha^{\alpha^2+2}}C^{\alpha^2+2}\\&=
B^{\mu_0}C^{\alpha^2+2}.
\end{align*}
Applying $\theta$ to the second and fifth terms of the right hand side above and then inverting, gives
$$
A^{-1}C^{\alpha^2+2}A^{\alpha(1+\alpha+\cdots+\alpha^{\alpha^2-1})}=C^{\alpha^2+2} A^{-\mu_0},
$$
so
$$
B^{\mu_0}C^{\alpha^2+2}=C^{\alpha^2+2}A^{-\mu_0}.
$$
Let $\alpha_0$ be the inverse of $\alpha$ modulo the order of $A$. We then have
$$
B^{\mu_0}=C^{\alpha^2+2}A^{-\mu_0}C^{-(\alpha^2+2)}=
A^{-\mu_0 \alpha_0^{\alpha^2+2}}.
$$
Conjugating this central element by $C^{\alpha^2+2}$ yields $B^{\mu_0}=A^{-\mu_0}$.
Conjugating $A^{\mu_0}\in Z(G)$ by $C$ and $B^{\mu_0}\in Z(G)$ by $C^{-1}$, we derive
$A^{\mu_0(\alpha-1)}=1=B^{\mu_0(\alpha-1)}$.
\end{proof}

Recalling the definition (\ref{gamita}) of $\gamma$, 
according to \cite[Eqs. 2.10, 2.11, and 2.12]{M}, we have
\begin{equation}
\label{gamita3}
A^\gamma B^\gamma=1,\; A^{(\alpha-1)\gamma}=1= B^{(\alpha-1)\gamma},\quad \alpha>0.
\end{equation}

We refer to $\alpha$ as 3-admissible if one of
the following 3 possibilities occurs: $v_3(\alpha-1)=0$; $v_3(\alpha-1)>1$; $v_3(\alpha-1)=1$ and $(\alpha-1)/3\equiv 1\mod 3$.

\begin{theorem}\label{nul} The Macdonald group $G(\alpha)$ is nilpotent of class at most 7.
\end{theorem}

\begin{proof} As indicated in \cite[pp. 610]{M}, we may assume that $\alpha>1$ and we make this assumption. Set
$$
f(\alpha)=\begin{cases} (\alpha-1)^2 & \text{ if }2\nmid (\alpha-1)\text{ and  }\alpha\text{ is 3-admissible},\\
(\alpha-1)^2/2 & \text{ if }2|(\alpha-1)\text{ and  }\alpha\text{ is 3-admissible},\\
3(\alpha-1)^2 & \text{ if }2\nmid (\alpha-1)\text{ and  }\alpha\text{ is not 3-admissible},\\
3(\alpha-1)^2/2 & \text{ if }2|(\alpha-1)\text{ and  }\alpha\text{ is not 3-admissible}.
\end{cases}
$$
We claim that $A^{f(\alpha)}B^{f(\alpha)}=1$. To see the claim, we will repeatedly and implicitly use
the fact \cite[Section 4]{M} that $|G(\alpha)|$ and $\alpha-1$ share the same prime factors. Now, 
if $\alpha$ is 3-admissible then $A^{f(\alpha)}B^{f(\alpha)}=1$ by (\ref{gamita3}) and Proposition \ref{calc1}.
Suppose next $\alpha$ is not 3-admissible, so that $\alpha=1+3k$, $k\in\N$, and $k\equiv -1\mod 3$. If $k=-1+3u$, $u\in\N$,
and $3\nmid u$, then $A^{f(\alpha)}B^{f(\alpha)}=1$ by (\ref{gamita3}) and Proposition \ref{calc1}. Assume next
$k\equiv -1\mod 9$. If $2\nmid (\alpha-1)$ then $A^{f(\alpha)}B^{f(\alpha)}=1$ by Propositions~\ref{calc2} and \ref{calc3}. 
Assume finally that $2|(\alpha-1)$. Then $A^{3^s(\alpha-1)^2/2}B^{3^s(\alpha-1)^2/2}=1$, $s\in\N$, by 
(\ref{gamita3}) and Proposition \ref{calc1}, and $A^{3(\alpha-1)^2 2^t}B^{3(\alpha-1)^2 2^t}=1$, $t\geq 0$, by 
Propositions \ref{calc2} and~\ref{calc3}. Since all factors are in $Z(G)$ and 
$\gcd(3^s(\alpha-1)^2/2,3(\alpha-1)^2 2^t)=3(\alpha-1)^2/2$, it follows that $A^{f(\alpha)}B^{f(\alpha)}=1$ also in this case.
This proves the claim. In particular, $A^{f(\alpha)},B^{f(\alpha)}\in Z(G)$, so
conjugating $A^{f(\alpha)}\in Z(G)$ by $C$ and $B^{f(\alpha)}\in Z(G)$ by $C^{-1}$, we obtain
$A^{(\alpha-1)f(\alpha)}=1=B^{(\alpha-1)f(\alpha)}$.

A careful calculation reveals that $\alpha^{f(\alpha)/(\alpha-1)}\equiv 1\mod f(\alpha)$, so
$C^{f(\alpha)/(\alpha-1)}\in Z_2(G)$. Next set
$$
g(\alpha)=\begin{cases} \alpha-1 & \text{ if }2\nmid (\alpha-1),\\
(\alpha-1)/2  & \text{ if }2\mid (\alpha-1),
\end{cases}\text{ and }\qquad h(\alpha)=f(\alpha)/g(\alpha).
$$
Another careful calculation shows that $1+\alpha+\cdots+\alpha^{h(\alpha)-1}\equiv h(\alpha)\mod f(\alpha)$,
which implies 
\begin{equation}
\label{ah}
\alpha(1+\alpha+\cdots+\alpha^{h(\alpha)-1})\equiv \alpha h(\alpha)\equiv h(\alpha)\mod f(\alpha).
\end{equation}
We have
$$
(A^{h(\alpha)})^B=(A^B)^{h(\alpha)}=(AC)^{h(\alpha)}=C^{h(\alpha)}A^{\alpha(1+\alpha+\cdots+\alpha^{h(\alpha)-1})}.
$$
Here $C^{h(\alpha)}\in Z_2(G)$, as $C^{f(\alpha)/(\alpha-1)}\in Z_2(G)$, and 
$A^{\alpha(1+\alpha+\cdots+\alpha^{h(\alpha)-1})}\equiv A^{h(\alpha)}\mod Z(G)$ by (\ref{ah}). It follows that
$A^{h(\alpha)}\in Z_3(G)$, and applying $\theta$ we deduce $B^{h(\alpha)}\in Z_3(G)$.

Suppose first that $\alpha$ is 3-admissible. Then $h(\alpha)=\alpha-1$. Since $A^C=A^{\alpha-1}A$
and ${}^C B=B^{\alpha-1}B$, with $A^{\alpha-1},B^{\alpha-1}\in Z_3(G)$, we infer $C\in Z_4(G)$, whence
$Z_5(G)=G$.

Suppose next that $\alpha$ is not 3-admissible. Then $h(\alpha)=3(\alpha-1)$. Moreover, $3|(\alpha-1)$,
so $\alpha^3\equiv 1\mod 3(\alpha-1)$, whence $C^3\in Z_4(G)$. We claim that $A^3\in Z_5(G)$. Indeed, we have
\begin{equation}
\label{ah2}
(A^{3})^B=(A^B)^{3}=(AC)^{3}=C^{3}A^{\alpha(1+\alpha+\alpha^2)},
\end{equation}
where $C^3\in Z_4(G)$. Moreover, we have $\alpha=1+3k$, with $k\in\Z$,
which readily gives
$1+\alpha+\alpha^2\equiv 3\mod 9$, and therefore $\alpha(1+\alpha+\alpha^2)\equiv 3\mod 9$.
Moreover, $\alpha=1+3k$ also implies $\alpha\equiv 1\mod k$, so $\alpha(1+\alpha+\alpha^2)\equiv 3\mod k$.
As $\gcd(3,k)=1$, it follows that $\alpha(1+\alpha+\alpha^2)\equiv 3\mod 9k$, that is,
$\alpha(1+\alpha+\alpha^2)\equiv 3\mod 3(\alpha-1)$. But $A^{3(\alpha-1)}\in Z_3(G)$, so (\ref{ah2})
implies $A^3\in Z_5(G)$, as claimed. Applying $\theta$, we deduce $B^3\in Z_5(G)$.
Since $A^C=A^{\alpha-1}A$
and ${}^C B=B^{\alpha-1}B$, with $A^{\alpha-1},B^{\alpha-1}\in Z_5(G)$ and $3|(\alpha-1)$, we infer $C\in Z_6(G)$, whence
$Z_7(G)=G$.
\end{proof}

\section{More valuation calculations}\label{secv2}

We adopt the following conventions for the remainder of the paper: $p\in\N$ stands for a prime factor of $\alpha-1$,
$m=v_p(\alpha-1)$, and $J=J(\alpha)$ is the Sylow $p$-subgroup of $G(\alpha)$. 
By Case 1 we mean that either $p>3$ or else $p=3$, where $m>1$ or $(\alpha-1)/3\equiv 1\mod 3$.
By Case 2 we understand that $p=2$. By Case 3 we signify that $p=3$, $m=1$, and $(\alpha-1)/3\equiv -1\mod 3$.

Consider the following polynomial expression in $\alpha$:
\begin{equation}
\label{gamitamedio}
\xi=2+\alpha+\cdots+\alpha^{-\alpha-1},\text{ if }\alpha<0.
\end{equation}

\begin{prop}\label{calc4} Suppose $\alpha<0$, $\alpha\neq -2$. Then
\begin{equation}\label{valchi}
v_p(\xi)=\begin{cases} 2m & \text{ in Case 1},\\
2m-1 & \text{ in Case 2},\\
3 & \text{ if }p=3, \alpha=1+3k, k=-1+3u,\text{ and }\gcd(3,u)=1.
\end{cases}
\end{equation}
\end{prop}

\begin{proof} If $\alpha=-1$ then $\xi=2$ and the result holds. Suppose henceforth that $\alpha\neq -1$.
Then $\alpha\leq -3$, so $(-\alpha)^{-\alpha}>2-\alpha$, whence $\xi\neq 0$. We have
$$
\xi=2+\alpha+\cdots+\alpha^{-\alpha-1}=1+\frac{\alpha^{-\alpha}-1}{\alpha-1}=\frac{\alpha-1+\alpha^{-\alpha}-1}{\alpha-1},
$$
where
$$
\alpha^{-\alpha}=(1+(\alpha-1))^{-\alpha}=1+(-\alpha)(\alpha-1)+{{-\alpha}\choose{2}}(\alpha-1)^2+{{-\alpha}\choose{3}}(\alpha-1)^3+
{{-\alpha}\choose{4}}(\alpha-1)^4+\cdots,
$$
so
$$
\alpha-1+\alpha^{-\alpha}-1=\alpha-1+(-\alpha)(\alpha-1)+{{-\alpha}\choose{2}}(\alpha-1)^2+{{-\alpha}\choose{3}}(\alpha-1)^3+
{{-\alpha}\choose{4}}(\alpha-1)^4+
\cdots,
$$
and therefore
\begin{equation}\label{chiqui}
\xi = [-1+{{-\alpha}\choose{2}}](\alpha-1)+{{-\alpha}\choose{3}}(\alpha-1)^2+{{-\alpha}\choose{4}}(\alpha-1)^3+
\cdots.
\end{equation}

Suppose first we are in Case 1. Write (\ref{chiqui}) in the form
$$
\xi=-\frac{[(\alpha-1)(\alpha+2)-1](\alpha+2)(\alpha-1)^2}{6}+{{-\alpha}\choose{4}}(\alpha-1)^3+
\cdots
$$
This yields (\ref{valchi}) when $p\neq 3$. Suppose next $p=3$, so that $\alpha=1+3^m k$, where $k\in\Z$
and $3\nmid k$. Then $\alpha+2=3(3^{m-1}k+1)$ and, by hypothesis, either $m>1$ or else $m=1$ and $k\equiv 1\mod 3$.
In both cases $v_3(\alpha+2)=1$, which completes the proof of (\ref{valchi}) in Case 1.

Suppose next that $p=2$. We see that (\ref{valchi}) holds by writing (\ref{chiqui}) in the form
$$
\xi=\frac{(\alpha+2)(\alpha-1)^2}{2}+{{-\alpha}\choose{3}}(\alpha-1)^2+{{-\alpha}\choose{4}}(\alpha-1)^3+
\cdots.
$$

Suppose finally that $p=3$, $\alpha=1+3k$, $k=-1+3u$, and $\gcd(3,u)=1$. Write (\ref{chiqui}) in the form
\begin{align*}
\xi &=\frac{[3-\alpha(\alpha+1)](\alpha+2)(\alpha-1)^2}{6}+{{-\alpha}\choose{4}}(\alpha-1)^3+{{-\alpha}\choose{5}}(\alpha-1)^4+\cdots\\
&=\frac{(4[3-\alpha(\alpha+1)]+\alpha(\alpha+1)(\alpha+3)(\alpha-1))}{8}(1+k)(\alpha-1)^2+{{-\alpha}\choose{5}}(\alpha-1)^4+\cdots.
\end{align*}
As $v_3(k+1)=1$, this proves (\ref{valchi}) in this case.
\end{proof}

As explained in \cite[pp. 606]{M}, we have
\begin{equation}\label{chigom}
A^\xi B^\xi=1,\quad \alpha<0.
\end{equation}
Conjugating the central elements $A^\xi$ and $B^\xi$ by
$C$ and $C^{-1}$, respectively, we obtain
\begin{equation}\label{chigom2}
A^{(\alpha-1)\xi}=1=B^{(\alpha-1)\xi},\quad \alpha<0.
\end{equation}

\section{A presentation of $J$}\label{secv3}

\begin{prop}\label{preex} Let $T$ be a group with presentation $\langle X\,|\,  R\rangle$, and let $\pi:F\to T$
be a corresponding epimorphism with kernel $\overline{R}$, the normal closure of $R$ in a free group $F$ on $X$.
Let $\sigma:T\to U$ be a group epimorphism, and let $V$ be any subset of $F$ such that $\langle V\rangle^\pi=\ker(\sigma)$.
Then $U$ has presentation $\langle X\,|\,  R\cup V\rangle$. 
\end{prop}

\begin{proof} Notice that $\ker(\pi\sigma)$ is the preimage of $\ker(\sigma)$ under $\pi$. As $\ker(\sigma)=\langle V\rangle^\pi$
and $\ker(\pi)=\overline{R}$, it follows that $\ker(\pi\sigma)=\overline{R}\langle V\rangle=\overline{R\cup V}$.
\end{proof}

\begin{cor}\label{precor} Let $T$ be a finite nilpotent group with presentation $\langle X\,|\,  R\rangle$, and let $\pi:F\to T$
be a corresponding epimorphism with kernel $\overline{R}$, the normal closure of $R$ in a free group $F$ on $X$.
Let $p\in\N$ be a prime, and let $P$ be the Sylow $p$-subgroup of $T$. For $x\in X$, let $n_x$ be any natural number
such that $x^{n_x}\in\overline{R}$, that is $(x^\pi)^{n_x}=1$, and set $m_x=v_p(n_x)$. Then $P$ has presentation 
$\langle X\,|\,  R\cup V\rangle$,
where $V=\{x^{p^{m_x}}\,|\, x\in X\}$.
\end{cor}

\begin{proof} As $T$ is a finite nilpotent group, we have $T=P\times Q$, where $Q$ is the direct product of all
other Sylow subgroups of $T$. Let $\sigma:T\to P$ and $\tau:T\to Q$ be projections corresponding to the decomposition 
$T=P\times Q$. Then $\langle V\rangle^\pi=Q=\ker(\sigma)$, so Proposition \ref{preex} applies.
\end{proof}

We next apply Corollary \ref{precor} when $T=G$ has presentation 
$\langle x,y\,|\, x^{[x,y]}=x^{\alpha},\, y^{[y,x]}=y^\alpha\rangle$.

\begin{theorem}\label{predela} The Sylow $p$-subgroup $J$ of $G$ has presentation
\begin{equation}
\label{sp1}
\langle x,y\,|\, x^{[x,y]}=x^\alpha,\, y^{[y,x]}=y^\alpha, x^{p^{3m}}=1, y^{p^{3m}}=1\rangle\text{ in Case 1},
\end{equation}
\begin{equation}
\label{sp2}
\langle x,y\,|\, x^{[x,y]}=x^\alpha,\, y^{[y,x]}=y^\alpha, x^{2^{3m-1}}=1, y^{2^{3m-1}}=1\rangle\text{ in Case 2},
\end{equation}
\begin{equation}
\label{sp3}
\langle x,y\,|\, x^{[x,y]}=x^\alpha,\, y^{[y,x]}=y^\alpha, x^{81}=1, y^{81}=1\rangle\text{ in Case 3}.
\end{equation}
\end{theorem}

\begin{proof} Suppose first we are in Case 1, Case  2, or else Case 3 with $\alpha=1+3k$,
$k=-1+3u$, and $\gcd(3,u)=1$. 
If $\alpha>0$ the result follows from (\ref{gamita3}), Proposition \ref{calc1}, and Corollary \ref{precor}. If $\alpha<0$ 
the result follows from (\ref{chigom2}), Proposition \ref{calc4}, and Corollary \ref{precor}.

Suppose next we are in Case 3, with $\alpha=1+3k$, $k=-1+9v$, and $v\in\Z$. The
result then follows from Proposition  \ref{calc2}, Proposition \ref{calc3}, and Corollary \ref{precor}.
\end{proof}

The given presentations
make it obvious that the isomorphism type of $J(\alpha)$ is invariant within the congruence class of $\alpha$ modulo~$p^{3m}$
for (\ref{sp1}), modulo $2^{3m-1}$ for (\ref{sp2}), and modulo 81 for (\ref{sp3}). 
We may thus assume without loss that $\alpha>0$ and we make this assumption for the remainder of the paper.

\section{An upper bound for the order of $J$}\label{secv4}

We see from Theorem \ref{predela} that $J$ is generated by elements $A$ and $B$ subject to the defining relations 
$A^{[A,B]}=A^\alpha$, $B^{[B,A]}=B^\alpha$, as well as $A^{p^{3m}}=1=B^{p^{3m}}$ in Case 1, 
$A^{2^{3m-1}}=1=B^{2^{3m-1}}$ in Case 2,
and $A^{81}=1=B^{81}$ in Case 3. We also have the additional relations $A^{p^{2m}}B^{p^{2m}}=1$ in Case 1 and
 $A^{2^{2m-1}}B^{2^{2m-1}}=1$ in Case 2, due to (\ref{gamita3}) and Proposition \ref{calc1}, as well as
$A^{27}B^{27}=1$ in Case 3, due to Propositions \ref{calc2} and \ref{calc3}. 
In any case, we set $C=[A,B]$. We proceed to derive further relations amongst $A$, $B$, and $C$, as well
as additional properties of $J$.

The defining relations of $J$ ensure the existence of an automorphism $A\leftrightarrow B$, $C\leftrightarrow C^{-1}$, of $J$.
Alternatively, use the fact that $J$ is a characteristic subgroup of $G$. 

\begin{lemma}\label{dec} Let $T$ be a group generated by elements $X,Y,Z$ of finite order satisfying:

$$X^Z=X^{u},\; Y^Z=Y^{v}\text{ for some integers }u,v;\; X^Y\in \langle X,Z\rangle;\; Y^X\in \langle Y,Z\rangle.
$$
Then $T$ is the product of the subgroups $\langle X\rangle,\langle Y\rangle, \langle Z\rangle$ in any of the six possible orderings.
In particular, $T$ is finite group whose order is a factor of $|X||Y||Z|$.
\end{lemma}

\begin{proof} As $Z$ normalizes $\langle X\rangle$ and $\langle Y\rangle$, we have
$$
\langle X,Z\rangle=\langle X\rangle\langle Z\rangle=\langle Z\rangle\langle X\rangle,\;
\langle Y,Z\rangle=\langle Y\rangle\langle Z\rangle=\langle Z\rangle\langle Y\rangle.
$$
By hypothesis, conjugation by $Y$ sends $\langle X\rangle$ into $\langle X\rangle\langle Z\rangle$; as conjugation by $Y$ preserves 
$\langle Z\rangle \langle Y\rangle$,
we see that conjugation by $Y$ preserves $\langle X\rangle\langle Z\rangle\langle Y\rangle$. Thus 
$Y^{-i} \langle X\rangle\langle Z\rangle\langle Y\rangle Y^{i}\subseteq 
\langle X\rangle\langle Z\rangle\langle Y\rangle$ for any $i\in\Z$, which implies 
$\langle Y\rangle\langle X\rangle\langle Z\rangle\langle Y\rangle\subseteq\langle X\rangle\langle Z\rangle\langle Y\rangle$.
As the reverse inclusion is clear, we have
$\langle Y\rangle\langle X\rangle\langle Z\rangle\langle Y\rangle=\langle X\rangle\langle Z\rangle\langle Y\rangle$.
This shows that $\langle X\rangle\langle Z\rangle\langle Y\rangle$ is closed under multiplication. As $\langle X\rangle\langle Z\rangle \langle Y\rangle$ is a finite subset of $T$, it is a subgroup of $T$. Since $\langle X\rangle\langle Z\rangle \langle Y\rangle$ contains $X$, $Y$ and $Z$,  we infer $T=\langle X\rangle\langle Z\rangle \langle Y\rangle$. 

Our hypotheses ensure that the roles of $X$ and $Y$ are interchangeable in the above argument, so $T$ is the product of 
$\langle X\rangle,\langle Z\rangle, \langle Y\rangle$ in any order. Since the order of $\langle X,Z\rangle=
\langle X\rangle\langle Z\rangle$ is a factor of $|X||Z|$, it follows that $|T|$ is a factor of $|X||Z||Y|$.
\end{proof}

\begin{lemma}\label{valdif} Let $a,b,k$ be integers such that $a\geq 1$ and $b\geq 0$, and suppose that $p$ is odd.
Then there is an integer $t$ such that
$$
(1+kp^a)^{p^b}=1+k p^{a+b}+kt p^{2a+b}.
$$
\end{lemma}

\begin{proof} This follows easily by induction on $b$.
\end{proof}

\begin{lemma}\label{valdif2} Let $a,b,k$ be integers with $a,b\geq 1$.
Then there is an integer $t$ such that
$$
(1+k2^a)^{2^b}=1+k 2^{a+b}+k^2 2^{2a+b-1}+kt 2^{2a+b}.
$$
\end{lemma}

\begin{proof} This follows easily by induction on $b$.
\end{proof}

\begin{prop}\label{masrex} Every element of $J$ can be written in the form  $A^iB^jC^k$, where
$0\leq i<i_0$, $0\leq j<j_0$, $0\leq k<k_0$, and $(i_0,j_0,k_0)=(p^{3m},p^{2m},p^{2m})$
in Case 1, in which case $A^{p^{3m}}=B^{p^{3m}}=C^{p^{2m}}=A^{p^{2m}}B^{p^{2m}}=1$ and $|J|\leq p^{7m}$,
$(i_0,j_0,k_0)=(2^{3m-1},2^{2m-1},2^{2m-1})$ 
in Case 2, in which case $A^{2^{3m-2}}=C^{2^{2m-1}}=B^{2^{3m-2}}$, $A^{2^{3m-1}}=B^{2^{3m-1}}=C^{2^{2m}}=A^{2^{2m-1}}B^{2^{2m-1}}=1$,
and $|J|\leq 2^{7m-3}$, and
$(i_0,j_0,k_0)=(81,27,27)$ 
in Case 3, in which case $A^{81}=B^{81}=C^{27}=A^{27}B^{27}=1$ and~$|J|\leq 3^{10}$.
\end{prop}

\begin{proof} Suppose first we are in Case 1. 
Then $A^{p^{2m}}B^{p^{2m}}=1$, so $A^{p^{2m}}\in Z(J)$. Therefore,
$$
A^{p^{2m}}=(A^{p^{2m}})^B=(A^B)^{p^{2m}}=(AC)^{p^{2m}}=
C^{p^{2m}}A^{\alpha(1+\alpha+\cdots+\alpha^{p^{2m}-1})}.
$$
Using Lemma \ref{valdif} with $\alpha=1+kp^m$ and $b=2m$, we see that $(\alpha^{p^{2m}}-1)/(\alpha-1)\equiv p^{2m}\mod p^{3m}$,
so $\alpha(\alpha^{p^{2m}}-1)/(\alpha-1)\equiv p^{2m}\mod p^{3m}$, and $C^{p^{2m}}=1$. The result now
follows from Lemma \ref{dec}.

Suppose next we are in Case 2. Then $A^{2^{2m-1}}B^{2^{2m-1}}=1$, so
$A^{2^{2m-1}}\in Z(J)$. Therefore,
$$
A^{2^{2m-1}}=(A^{2^{2m-1}})^B=(A^B)^{2^{2m-1}}=(AC)^{2^{2m-1}}=
C^{2^{2m-1}}A^{\alpha(1+\alpha+\cdots+\alpha^{2^{2m-1}-1})}.
$$
Using Lemma \ref{valdif2} with $\alpha=1+k2^m$, with $k$ odd, and $b=2m-1$, 
we see that $(\alpha^{2^{2m-1}}-1)/(\alpha-1)\equiv 2^{2m-1}+2^{3m-2}\mod 2^{3m-1}$,
whence $\alpha(\alpha^{2^{2m-1}}-1)/(\alpha-1)\equiv 2^{2m-1}+2^{3m-2}\mod 2^{3m-1}$, so $C^{2^{2m-1}}A^{2^{3m-2}}=1$. 
As $A^{2^{3m-2}}$ is its own inverse, it is equal to both $C^{2^{2m-1}}$ and $B^{2^{3m-2}}$, so $C^{2^{2m}}=1$. The result now
follows from Lemma \ref{dec}.

Suppose finally we are in Case 3.
Then $A^{27}B^{27}=1$, so $A^{27}\in Z(J)$. Therefore,
$$
A^{27}=(A^{27})^B=(A^B)^{27}=(AC)^{27}=
C^{27}A^{\alpha(1+\alpha+\cdots+\alpha^{26})}.
$$
Using Lemma \ref{valdif} with $\alpha=1+3k$ and $b=3$, we see that $(\alpha^{27}-1)/(\alpha-1)\equiv 27\mod 81$,
whence $\alpha(\alpha^{27}-1)/(\alpha-1)\equiv 27\mod 81$, so $C^{27}=1$. The result now
follows from Lemma \ref{dec}.
\end{proof}

\section{Order and basic properties of $J$}\label{secv6}

\begin{theorem}\label{main1} Every element of $J$ can be written uniquely in the form  $A^iB^jC^k$, where
$0\leq i<i_0$, $0\leq j<j_0$, $0\leq k<k_0$, and $(i_0,j_0,k_0)=(p^{3m},p^{2m},p^{2m})$
in Case 1, in which case $|J|=p^{7m}$,
$(i_0,j_0,k_0)=(2^{3m-1},2^{2m-1},2^{2m-1})$ 
in Case 2, in which case $|J|=2^{7m-3}$, and
$(i_0,j_0,k_0)=(81,27,27)$ in Case 3, in which case $|J|=3^{10}$.
\end{theorem}

\begin{proof} Immediate consequence of Proposition \ref{masrex} and Theorems \ref{model}, \ref{model2}, and \ref{model3}.
\end{proof}

In Proposition \ref{orden} below, the automorphism $A\leftrightarrow B$ of $J$ yields the corresponding
results when we interchange the roles of $A$ and $B$.

\begin{prop}\label{orden} (a) $A$ and $C$ have respective orders $p^{3m}$ and  $p^{2m}$ in Case 1,
$2^{3m-1}$ and $2^{2m}$ in Case 2, and $81$ and $27$ in Case 3. Moreover, 
$A^{2^{3m-2}}=C^{2^{2m-1}}$ has order 2 in Case 2.
 
(b) In Case 1, $\langle A,C\rangle=\langle A\rangle\rtimes \langle C\rangle$ has order $p^{5m}$.
In Case 2, $\langle A\rangle\cap \langle C\rangle=\langle A^{2^{3m-2}}\rangle$ has order~2, and
$\langle A, C\rangle$ has order $2^{5m-2}$, with defining relations
$A^{2^{3m-1}}=1, A^{2^{3m-2}}=C^{2^{2m-1}}, A^C=A^\alpha$.
In  Case 3, $\langle A,C\rangle=\langle A\rangle\rtimes \langle C\rangle$ has order $3^7$.

(c) The group $\langle A\rangle\cap\langle B\rangle=\langle A,C\rangle\cap \langle B\rangle$ is 
equal to $\langle A^{p^{2m}}\rangle$ and has order $p^{m}$ in Case 1, to $\langle A^{{2^{2m-1}}}\rangle$
and has order $2^m$ in Case 2, and to $\langle A^{27}\rangle$ and has order 3 in Case 3.
\end{prop}

\begin{proof} (a) Proposition \ref{masrex} and the uniqueness statement of Theorem \ref{main1} imply 
that $A$ and $C$ have the stated order.

(b) Clearly $\langle C\rangle$ normalizes $\langle A\rangle$. Suppose first we are in Cases 1 or 3.
That $\langle A\rangle\cap \langle C\rangle$ is trivial
follows from part (a) and the uniqueness statement of Theorem \ref{main1}. That $\langle A,C\rangle=\langle A\rangle\rtimes \langle C\rangle$
has order $p^{5m}$ in Case 1 and order $3^7$ in Case 3 follows from part (a). Suppose next we are in Case 2.
That $\langle A\rangle\cap \langle C\rangle=\langle A^{2^{3m-2}}\rangle$ has order 2 follows from Proposition \ref{masrex} and
the uniqueness statement of Theorem~\ref{main1}. Thus,
$\langle A, C\rangle=\langle A\rangle \langle C\rangle$ has order $2^{5m-2}$ and satisfies the stated relations.
Any group generated by elements $A$ and $C$ satisfying the stated relations has order $\leq 2^{5m-2}$, so these
are defining relations.

(c) Proposition \ref{masrex} and the uniqueness statement of Theorem \ref{main1} imply
that $\langle A\rangle\cap\langle B\rangle$ is equal to $\langle A^{p^{2m}}\rangle$ in Case 1,
to $\langle A^{2^{2m-1}}\rangle$ in Case 2, and to $\langle A^{27}\rangle$ in Case 3. 
We have $J=\langle A,C\rangle\langle B\rangle$ by Lemma 
\ref{dec}, where $|J|=p^{7m}$, $|\langle A,C\rangle|=p^{5m}$, $|\langle B\rangle|=p^{3m}$ in Case 1,
$|J|=2^{7m-3}$, $|\langle A,C\rangle|=2^{5m-2}$, $|\langle B\rangle|=2^{3m-1}$ in Case 2, and
$|J|=3^{10}$, $|\langle A,C\rangle|=3^{7}$, $|\langle B\rangle|=3^{4}$ in Case 3.
Thus, $\langle A,C\rangle\cap \langle B\rangle$ must have order $p^{m}$ in Case~1,  $2^m$ in Case 2,
and $3$ in Case 3, and is therefore equal to $\langle A\rangle\cap\langle B\rangle$.
\end{proof}

\begin{prop}\label{centra} 
We have $C_J(A)=\langle A\rangle, C_J(B)=\langle B\rangle$ and $N_J(A)=\langle A,C\rangle, N_J(B)=\langle B,C\rangle$.
\end{prop}

\begin{proof} Let $x\in N_J(B)$. By Lemma \ref{dec}, we have $x=B^j A^i C^k$ for some $i,j,k\in\N$ and $i\geq 2$.
Then
$$
B^s=B^x=B^{B^j A^i C^k}=B^{A^i C^k}= (B A^{(\alpha-1)(\alpha+2\alpha^2+\cdots+(i-1)\alpha^{i-1})}C^{-i})^{C^k},
$$
for some $s\in\N$, so
$$
B^s=B^{\beta^k} A^{(\alpha-1)(\alpha+2\alpha^2+\cdots+(i-1)\alpha^{i-1})\alpha^k}C^{-i},
$$
where $\beta$ is the inverse of $\alpha$ modulo the order of $B$. Therefore
$$
B^{s-\beta^k}=A^{(\alpha-1)(\alpha+2\alpha^2+\cdots+(i-1)\alpha^{i-1})\alpha^k}C^{-i}
$$
is in $\langle A,C\rangle\cap \langle B\rangle$, which equals $\langle A\rangle\cap \langle B\rangle$ by Proposition
\ref{orden}. Thus $C^{-i}$ is in $\langle A\rangle\cap \langle C\rangle$.

Suppose first we are in Case 1. Then $\langle A\rangle\cap \langle C\rangle$ is trivial and
hence $i\equiv 0\mod p^{2m}$ by Proposition~\ref{orden}. As $A^{p^{2m}}B^{p^{2m}}=1$, we infer
$x=B^\ell C^k$, $\ell\in\N$, which proves that $N_J(B)=\langle B,C\rangle$.

Suppose next we are in Case 2. Then $\langle A\rangle\cap \langle C\rangle=\langle C^{2^{2m-1}}\rangle$ and
hence $i\equiv 0\mod 2^{2m-1}$ by Proposition \ref{orden}. As $A^{2^{2m-1}}B^{2^{2m-1}}=1$, we deduce
$x=B^\ell C^k$, $\ell\in\N$, so $N_J(B)=\langle B,C\rangle$.

Suppose finally we are in Case 3. Then $\langle A\rangle\cap \langle C\rangle$ is trivial and
hence $i\equiv 0\mod 27$ by Proposition \ref{orden}. As $A^{27}B^{27}=1$, it follows that
$x=B^\ell C^k$, $\ell\in\N$, and $N_J(B)=\langle B,C\rangle$.

The automorphism $A\leftrightarrow B$ now yields  $N_J(A)=\langle A,C\rangle$.

Assume now that $x\in C_J(B)$. By above, $x=B^\ell C^k$, where $k,\ell\in\N$, so $C^k\in C_J(B)$.

Suppose first we are in Case 1. Then $\alpha^k\equiv 1\mod p^{3m}$. It follows from Lemma \ref{valdif} 
that the order of $\alpha$ modulo $p^{3m}$ is $p^{2m}$, so $k\equiv 0\mod p^{2m}$, whence
$x\in \langle B\rangle$.

Suppose next we are in Case 2. Then $\alpha^k\equiv 1\mod 2^{3m-1}$. By Lemma \ref{valdif2},
the order of $\alpha$ modulo $2^{3m-1}$ is $2^{2m-1}$, so $k\equiv 0\mod 2^{2m-1}$. As
$C^{2^{2m-1}}=B^{2^{3m-2}}$, we infer $x\in \langle B\rangle$.

Suppose finally we are in Case 3. Then $\alpha^k\equiv 1\mod 81$. We deduce from Lemma \ref{valdif} 
that the order of $\alpha$ modulo $81$ is $27$, and therefore $k\equiv 0\mod 27$, so
$x\in \langle B\rangle$.

Thus $C_J(B)=\langle B\rangle$ in all cases.
The automorphism $A\leftrightarrow B$ now yields  $C_J(A)=\langle A\rangle$.
\end{proof}

\section{Upper and lower central series of $J$}\label{secv7}

The next two results describe the upper and lower central series of $J$ in all cases.

\begin{theorem}\label{main2} (a) In Case 1, the nilpotency class of $J$ is 5, and
$$
Z_1(J)=\langle A^{p^{2m}}\rangle, Z_2(J)=\langle A^{p^{2m}},C^{p^{m}}\rangle,
Z_3(J)=\langle  A^{p^{m}}, B^{p^{m}}, C^{p^{m}}\rangle, 
Z_4(J)=\langle  A^{p^{m}}, B^{p^{m}}, C\rangle.
$$

(b) In Case 2, if $m>1$, the nilpotency class of $J$ is 5, and
$$
Z_1(J)=\langle A^{2^{2m-1}}\rangle, Z_2(J)=\langle  A^{2^{2m-1}},  C^{2^{m-1}}\rangle,
Z_3(J)=\langle A^{2^{m}},  B^{2^{m}}, C^{2^{m-1}}\rangle, 
Z_4(J)=\langle  A^{2^{m-1}},  B^{2^{m-1}}, C\rangle.
$$

(c) In Case 3, the nilpotency class of $J$ is 7, and
$$
Z_1(J)=\langle A^{27}\rangle,\; Z_2(J)=\langle A^{27},C^{9}\rangle,\; Z_3(J)=\langle  A^{9}, B^{9}, C^{9}\rangle,
$$
$$
Z_4(J)=\langle  A^{9}, B^{9}, C^{3}\rangle,\;
Z_5(J)=\langle  A^{3}, B^{3}, C^3\rangle,\; Z_6(J)=\langle A^{3}, B^{3},  C\rangle.
$$

Moreover, in Cases 1 and 3, the terms of the upper and lower central series of $J$ coincide, in reverse order,
while in Case 2, if $m>1$, we have
$$
\gamma_2(J)=\langle  A^{2^{m}},  B^{2^{m}}, C\rangle, \gamma_3(J)=\langle  A^{2^{m}},  B^{2^{m}}, C^{2^{m}}\rangle, 
\gamma_4(J)=\langle  A^{2^{2m}},  [A,B^{2^{m}}]\rangle,\gamma_5(J)=\langle  A^{2^{2m}}\rangle.
$$
\end{theorem}

\begin{proof} We have $Z(J)=C_J(A)\cap C_J(B)=\langle A\rangle\cap \langle B\rangle$
by Proposition \ref{centra}. Here $\langle A\rangle\cap \langle B\rangle$ is equal to
$\langle A^{p^{2m}}\rangle$ in Case 1, to $\langle A^{2^{2m-1}}\rangle$ in Case 2,
and to $\langle A^{27}\rangle$ in Case 3, by Proposition \ref{orden}.

Set $H=J/Z(J)$. Let $a,b,c\in H$ be the images of $A,B,C$ under the canonical projection. 
Then $a^{p^{2m}}=b^{p^{2m}}=c^{p^{2m}}=1$ in Case 1, $a^{2^{2m-1}}=b^{2^{2m-1}}=c^{2^{2m-1}}=1$ in Case 2,
and $a^{27}=b^{27}=c^{27}=1$ in Case 3. Moreover, by Proposition \ref{orden}, 
the order of $H$ is $p^{6m}$ in Case 1, $2^{6m-3}$ in Case 2, and~$3^9$
in Case~3. Furthermore, by Lemma \ref{dec}, 
every element of $H$ can be written as a product of elements from 
$\langle a\rangle, \langle b\rangle$, and $\langle c\rangle$, in any fixed order. This implies uniqueness of expression
and that the given upper bounds for the orders of $a,b,c$ are the actual orders of these elements.

We claim that $C_H(b)$ is equal to $\langle b,c^{p^{m}}\rangle$ in Case 1, $\langle b,c^{2^{m-1}}\rangle$
in Case 2, and $\langle b,c^{9}\rangle$ in Case 3. Indeed, 
let $x\in C_H(b)$. Then $x=b^j a^i c^k$ for some $i,j,k\in\N$ with $i\geq 2$. The same argument used in the 
proof of Proposition \ref{centra} now yields $i\equiv 0\mod p^{2m}$ in Case 1, $i\equiv 0\mod 2^{2m-1}$ in Case 2,
and $i\equiv 0\mod 27$ in Case 3, so $x=b^j c^k$, which now implies
$\alpha^k\equiv 1\mod p^{2m}$ in Case 1, $\alpha^k\equiv 1\mod 2^{2m-1}$ in Case 2, and
$\alpha^k\equiv 1\mod 27$ in Case 3. In Case 1 the order of $\alpha$ modulo $p^{2m}$ is $p^{m}$,
in Case 2 the order of $\alpha$ modulo $2^{2m-1}$ is $2^{m-1}$, and in Case 3 the order of $\alpha$ modulo $27$ is $9$,
by Lemmas \ref{valdif} and \ref{valdif2}. So 
$k\equiv 0\mod p^{m}$ in Case 1, $k\equiv 0\mod 2^{m-1}$ in Case 2, and $k\equiv 0\mod 9$ in Case 3, 
and therefore $x$ is in $\langle b,c^{p^{m}}\rangle$ in Case 1, $\langle b,c^{2^{m-1}}\rangle$
in Case 2, and $\langle b,c^{9}\rangle$ in Case 3. This proves one inclusion in every case.
By above, $\alpha^{p^{m}}\equiv 1\mod p^{2m}$ in Case 1, $\alpha^{2^{m-1}}\equiv 1\mod 2^{2m-1}$
in Case 2, and $\alpha^{9}\equiv 1\mod 27$, so the reverse inclusion is clear, which proves the claim.
The automorphism $a\leftrightarrow b$ of $H$ now yields that  $C_H(a)$ is equal to 
$\langle a,c^{p^{m}}\rangle$ in Case~1, $\langle a,c^{2^{m-1}}\rangle$
in Case 2, and $\langle a,c^{9}\rangle$ in Case 3. Thus
$Z(H)$ is equal to $\langle a,c^{p^{m}}\rangle\cap\langle b,c^{p^{m}}\rangle$ in Case 1,
$\langle a,c^{2^{m-1}}\rangle\cap\langle b,c^{2^{m-1}}\rangle$ in Case 2, and 
$\langle a,c^{9}\rangle\cap\langle b,c^{9}\rangle$ in Case 3. The normal form of the elements of $H$
forces this intersection to be $\langle c^{p^{m}}\rangle$ in Case 1, $\langle c^{2^{m-1}}\rangle$ in Case 2,
and $\langle c^{9}\rangle$ in Case~3. The preimage of this group under the canonical
projection $J\to H$ is $\langle A^{p^{2m}},C^{p^{m}}\rangle$ in Case~1, $\langle A^{2^{2m-1}},C^{2^{m-1}}\rangle$ 
in Case 2, and $\langle A^{27},C^{9}\rangle$ in Case 3, which confirms the stated description of $Z_2(J)$ in all cases.

Next set $K=J/Z_2(J)$. Let $u,v,w\in K$ be the images of $A,B,C$ under the canonical projection.
Then $u^{p^{2m}}=v^{p^{2m}}=w^{p^{m}}=1$ in Case 1, $u^{2^{2m-1}}=v^{2^{2m-1}}=w^{2^{m-1}}=1$ in Case 2,
and $u^{27}=v^{27}=w^{9}=1$ in Case 3. Moreover, by Proposition \ref{orden}, 
the order of $K$ is $p^{5m}$ in Case 1, $2^{5m-3}$ in Case~2, and~$3^8$
in Case~3. Furthermore, by Lemma \ref{dec}, 
every element of $K$ can be written as a product of elements from 
$\langle u\rangle, \langle v\rangle$, and $\langle w\rangle$, in any fixed order. This implies uniqueness of expression
and that the given upper bounds for the orders of $u,v,w$ are the actual orders of these elements.

We claim that $C_K(v)$ is equal to $\langle v,u^{p^{m}}\rangle$ in Case 1, $\langle v,u^{2^{m}}\rangle$
in Case 2, and $\langle v,u^{9}\rangle$ in Case 3. Indeed, 
let $x\in C_K(v)$. Then $x=v^j u^i w^k$ for some $i,j,k\in\N$ with $i\geq 2$. The same argument used in the 
proof of Proposition \ref{centra} now yields $i\equiv 0\mod p^{m}$ in Case 1, $i\equiv 0\mod 2^{m-1}$ in Case 2,
and $i\equiv 0\mod 9$ in Case 3. Moreover, in Case 2, we also obtain
$$
(i-1)i/2\equiv 1+2+\cdots+i-1\equiv \alpha+2\alpha^2+\cdots+(i-1)\alpha^{i-1}\equiv 0\mod 2^{m-1}.
$$
We use for the first time that $m>1$. As $i\equiv 0\mod 2^{m-1}$, then $i$ is even. Since
$(i-1)i/2\equiv 0\mod 2^{m-1}$ and $i-1$ is odd, we infer $i\equiv 0\mod 2^{m}$. On the other hand, in Case 1, we have
$$
(u^{p^m})^v=(u^v)^{p^m}=(uw)^{p^m}=w^{p^m} u^{\alpha(1+\alpha+\cdots+\alpha^{p^m-1})}=u^{p^m}
$$
since $\alpha(\alpha^{p^m}-1)/(\alpha-1)\equiv p^m\mod p^{2m}$ by Lemma \ref{valdif}; in Case 2, we have
$$
(u^{2^m})^v=(u^v)^{2^m}=(uw)^{2^m}=w^{2^m} u^{\alpha(1+\alpha+\cdots+\alpha^{2^m-1})}=u^{2^m},
$$
since $\alpha(\alpha^{2^m}-1)/(\alpha-1)\equiv 2^m\mod 2^{2m-1}$ by Lemma \ref{valdif2}; and 
in Case~3, we have
$$
(u^{9})^v=(u^v)^{9}=(uw)^{9}=w^{9} u^{\alpha(1+\alpha+\cdots+\alpha^{8})}=u^{9}
$$
since $\alpha(\alpha^{9}-1)/(\alpha-1)\equiv 9\mod 27$ by Lemma \ref{valdif}. Thus $C_K(v)$ contains
$u^{p^m}$ in Case 1, $u^{2^m}$ in Case~2, and $u^{9}$ in Case 3. It follows that $w^k\in C_K(v)$
in all cases, which implies $\alpha^k\equiv 1\mod p^{2m}$ in Case~1, $\alpha^k\equiv 1\mod 2^{2m-1}$ in Case 2,
and $\alpha^k\equiv 1\mod 27$ in Case 3. Exactly as above, this implies 
$k\equiv 0\mod p^{m}$ in Case 1, $k\equiv 0\mod 2^{m-1}$ in Case 2, and $k\equiv 0\mod 9$ in Case 3.
This proves the claim. The automorphism $u\leftrightarrow v$ of $K$ now yields that $C_K(u)$ is equal to 
$\langle u,v^{p^{m}}\rangle$ in Case 1, $\langle u,v^{2^{m}}\rangle$
in Case 2, and $\langle u,v^{9}\rangle$ in Case 3. Thus
$Z(H)$ is equal to $\langle u,v^{p^{m}}\rangle\cap\langle v,u^{p^{m}}\rangle$ in Case 1,
$\langle u,v^{2^{m}}\rangle\cap\langle v,u^{2^{m-1}}\rangle$ in Case 2, and 
$\langle u,v^{9}\rangle\cap\langle v,u^{9}\rangle$ in Case 3. The normal form of the elements of $K$
forces this intersection to be $\langle u^{p^{m}},v^{p^{m}}\rangle$ in Case 1, $\langle u^{2^{m}},v^{2^{m}}\rangle$ in Case 2,
and $\langle u^{9},v^{9}\rangle$ in Case 3. The preimage of this group under the canonical
projection $J\to K$ is $\langle A^{p^{m}},B^{p^{m}},C^{p^{m}}\rangle$ in Case 1, $\langle A^{2^{2m}},B^{2^{2m}},C^{2^{m-1}}\rangle$ 
in Case 2, and $\langle A^{9},B^{9},C^{9}\rangle$ in Case 3, which confirms the stated description of $Z_3(J)$ in all cases.

Suppose next we are in Case 1, and let 
$$M=\langle x,y,z\,|\, x^{p^{m}}=1, y^{p^{m}}=1, z^{p^{m}}=1, [x,y]=z, [x,z]=1, [y,z]=1\rangle\cong\mathrm{Heis}(\Z/p^m\Z).$$ 
Consider the assignment $A\mapsto x$, $B\mapsto y$. The defining relations
of $J$ are preserved, which yields a group epimorphism $f:J\to M$. As $M$ has exponent $p^{m}$, we have 
$J^{p^{m}}\subseteq \ker(f)$. Recalling that $Z_3(J)=\langle A^{p^{m}},B^{p^{m}},C^{p^{m}}\rangle$, it follows that
$Z_3(J)\subseteq J^{p^{m}}$. On the other hand, if $A^i B^j C^k\in\ker(f)$, then the normal form of the elements of $M$ implies that all of $i,j,k$ are
multiples of $p^{m}$, so $\ker(f)\subseteq Z_3(J)$. This proves that $Z_3(J)=J^{p^{m}}=\ker(f)$, 
whence $J/Z_3(J)\cong M$. As $Z(M)=\langle z\rangle$,
it follows that $Z_4(J)=\langle A^{p^{m}},B^{p^{m}},C\rangle$. Given that $C\in Z_4(J)$, we deduce $Z_5(J)=J$.

Suppose next we are in Case 2, and let 
$$M=\langle x,y,z\,|\, x^{2^m}=1, y^{2^m}=1, z^{2^{m-1}}=1, [x,y]=z, [x,z]=1, [y,z]=1\rangle,$$
a quotient of the
Heisenberg group over $\Z/2^m\Z$ by the $2^{m-1}$th power of its center.
Consider the assignment $A\mapsto x$, $B\mapsto y$. The defining relations
of $J$ are preserved, which yields a group epimorphism $f:J\to M$.  Recalling that
$Z_3(J)=\langle A^{2^{m}},B^{2^{m}},C^{2^{m-1}}\rangle$, we deduce
$Z_3(J)\subseteq \ker(f)$. Moreover, if $A^i B^j C^k\in\ker(f)$, then the normal form of the elements of 
$M$ implies that $i,j$ are multiples of $2^m$ and $k$ is a multiple of $2^{m-1}$, which proves $Z_3(J)=\ker(f)$, 
whence $J/Z_3(J)\cong M$. As the nilpotency class of $M$ is 2, that of $J$ is 5. In fact,
as $Z(M)=\langle x^{2^{m-1}}, y^{2^{m-1}}, z\rangle$,
it follows that $Z_4(J)=\langle A^{2^{m-1}},B^{2^{m-1}},C\rangle$. Given that $C\in Z_4(J)$, we deduce $Z_5(J)=J$.

Suppose next we are in Case 3, and set $L=J/Z_3(J)$. As $|Z(K)|=9$, we have $|L|=3^{6}$. Repeating with $L$ the analysis made with $H$,
we find $|Z(L)|=3$ and $Z_4(J)=\langle A^{9},B^9,C^{3}\rangle$. Next set $T=J/Z_4(J)$. Since $|Z(L)|=3$,
we have $|T|=3^{5}$. Mimicking  with $T$ the argument used with~$K$,
we find $|Z(T)|=9$, $Z_5(J)=\langle A^{3},B^3,C^{3}\rangle$, and $|J/Z_5(J)|=3^3$.
Let 
$$M=\langle r,s,t\,|\, r^{3}=1, s^{3}=1, t^{3}=1, [r,s]=t, [r,t]=1, [s,t]=1\rangle\cong \mathrm{Heis}(\Z/3\Z).$$
Then $A\mapsto r$, $B\mapsto s$  extends to an epimorphism
$f:J\to M$.  As $M$ has exponent $3$, we have 
$J^{3}\subseteq \ker(f)$. But $Z_5(J)=\langle A^{3},B^{3},C^{3}\rangle$, so $Z_5(J)\subseteq J^{3}$. 
As $|J/Z_5(J)|=3^3$, we infer $Z_5(J)=J^{3}=\ker(f)$ and $J/Z_5(J)\cong M$.
Since $Z(M)=\langle t\rangle$,
it follows that $Z_6(J)=\langle A^{3},B^{3},C\rangle$. Given that $C\in Z_6(J)$, we deduce $Z_7(J)=J$.

It remains to show that the lower central series of $J$ is as stated. We will make use of the formulas:
\begin{equation}
\label{conmutador}
[A,B^{i}]=B^{(\alpha-1)(\alpha+2\alpha^2+\cdots+(i-1)\alpha^{i-1})}C^{i},\;
[B,A^{i}]=A^{(\alpha-1)(\alpha+2\alpha^2+\cdots+(i-1)\alpha^{i-1})}C^{-i},\;
i\geq 2.
\end{equation}
Suppose first we are in Case 1 (resp. Case 3). 
Since $J=Z_5(J)$ (resp. $J=Z_7(J)$), we infer $\gamma_2(J)\subseteq Z_4(J)$ (resp. $\gamma_2(J)\subseteq Z_6(J)$).
But $[A,B]=C$, $[A,C]=A^{\alpha-1}$, and $[B,C^{-1}]=B^{\alpha-1}$, so $A^{p^m},B^{p^m},C\in \gamma_2(J)$. 
Thus $Z_4(J)=\gamma_2(J)$ (resp. $Z_6(J)=\gamma_2(J)$). It follows that 
$\gamma_3(J)\subseteq Z_3(J)$ (resp. $\gamma_3(J)\subseteq Z_5(J)$). As above, $A^{p^m},B^{p^m}\in \gamma_3(J)$.
Due to (\ref{conmutador}), we see that $C^{p^m}\in \gamma_3(J)$, so $\gamma_3(J)=Z_3(J)$ (resp. $\gamma_3(J)=Z_5(J)$). This implies 
$\gamma_4(J)\subseteq Z_2(J)$ (resp. $\gamma_4(J)\subseteq Z_4(J)$). Making use of Lemma \ref{valdif}, we find
that $v_p(\alpha^{p^m}-1)=2m$, which gives $A^{p^{2m}},B^{p^{2m}}\in \gamma_4(J)$. As
$$
\alpha+2\alpha^2+\cdots+(p^{m}-1)\alpha^{p^{m}-1}\equiv 1+2+\cdots+(p^{m}-1)\equiv 0\mod p^m,
$$
the case $i=p^{m}$ of (\ref{conmutador}) ensures that $C^{p^{m}}\in \gamma_4(J)$, 
so $\gamma_4(J)=Z_2(J)$ (resp. $\gamma_4(J)=Z_4(J)$).  We infer 
$\gamma_5(J)\subseteq Z(J)$ (resp. $\gamma_5(J)\subseteq Z_3(J)$). As above, $A^{p^{2m}},B^{p^{2m}}\in \gamma_5(J)$,
and the case $i=p^{2m}$ of (\ref{conmutador}) ensures that $C^{p^{2m}}\in \gamma_5(J)$. Thus 
$\gamma_5(J)=Z(J)$ (resp. $\gamma_5(J)=Z_3(J)$). This completes the proof in Case 1.

Suppose next we are in Case 3. From $\gamma_5(J)=Z_3(J)$, we infer $\gamma_6(J)\subseteq Z_2(J)$.
As $v_3(\alpha^9-1)=27$, we deduce $A^{27}\in \gamma_6(J)$.
Since $\alpha\equiv -2\mod 9$, we see that $\alpha+2\alpha^2+\cdots+8\alpha^{8}\equiv 0\mod 9$. Thus (\ref{conmutador})
yields $C^{9}\in  \gamma_6(J)$, which implies
$\gamma_6(J)=Z_2(J)$. Thus $\gamma_7(J)\subseteq Z(J)$, with $A^{27}\in \gamma_7(J)$ as above,
so $\gamma_7(J)=Z(J)$. This completes the proof in Case 3.

Suppose finally we are in Case 2. For any $i\geq 1$, $\gamma_{i+1}(J)$ is the normal subgroup generated by the set of all possible brackets between $\{A,B\}$ and any fixed generating $S_i$ subset of $\gamma_{i}(J)$ \cite[5.1.7]{Ro}. Taking $S_1=\{A,B\}$,
we find that $\gamma_2(J)=\langle A^{2^{m}}, B^{2^{m}}, C\rangle$, which is the normal closure of~$\{C\}$. Alternatively,
since $\langle A^{2^{m}}, B^{2^{m}}, C\rangle$
is a normal subgroup of $J$ contained in $[J,J]$ with abelian quotient, it must be equal to it. In both cases, normality
follows from (\ref{conmutador}). Taking $S_2=\{A^{2^m}, B^{2^{m}}, C\}$, we find that 
$\gamma_3(J)=\langle A^{2^{m}}, B^{2^{m}}, C^{2^m}\rangle$, where normality 
follows from (\ref{conmutador}) and Lemma \ref{valdif2}.
Taking $S_3=\{A^{2^m}, B^{2^{m}}, C^{2^m}\}$, we find that $\gamma_4(J)=\langle A^{2^{2m}},[A,B^{2^{m}}]\rangle$, 
where $[A,B^{2^{m}}]=B^{\ell 2^{2m-1}}C^{2^m}$, with $\ell$ odd,  by (\ref{conmutador}). As $A^{2^{2m}}, B^{2^{2m-1}}\in Z(J)$
and $\langle A^{2^{2m}}\rangle=\langle B^{2^{2m}}\rangle$, the normality of $\langle A^{2^{2m}},[A,B^{2^{m}}]\rangle$ is ensured,
as well as the fact that $\gamma_5(J)=\langle A^{2^{2m}}\rangle$.
\end{proof}

 \begin{prop}\label{mau} Suppose $m=1$ and $p=2$. Then $J$ is isomorphic to the generalized quaternion group of order~16,
$Z(J)=\langle A^2\rangle=\gamma_3(J)$, $Z_2(J)=\langle A^2,C\rangle=\gamma_2(J)$, the nilpotency class of $J$ is 3,
and the exponent of $J$ is 8.
\end{prop}

\begin{proof} The first part of the proof of Theorem \ref{main2} shows that
$Z(J)=\langle A^2\rangle$ and $Z_2(J)=\langle A^2,C\rangle$. The latter is a subgroup of $J$ of order 4
by Proposition \ref{orden} and $|J|=16$ by Theorem~\ref{main1}, so $J/Z_2(J)$ is abelian and therefore
$Z_3(J)=J$. Thus the nilpotency class of $J$  is 3. From $Z_3(J)=J$ we deduce $\gamma_2(J)\subseteq Z_2(J)$.
But $C,A^2\in\gamma_2(J)$, so $\gamma_2(J)=Z_2(J)$. This implies $\gamma_3(J)\subseteq Z(J)$.
As $A^2\in\gamma_3(J)$, we infer $\gamma_3(J)=Z(J)$. Since $v_2(\alpha-1)=1$, 
Theorem \ref{predela} ensures that $J$ has presentation
$$
\langle A,B\,|\, A^{[A,B]}=A^{-1},\; B^{[B,A]}=B^{-1},\; A^{4}=1=B^{4}\rangle,
$$
which, in accordance with the discussion following Theorem \ref{predela}, is independent of $\alpha$, as
all integers of the form $1+2k$, with $k$ odd, are congruent modulo 4. Set $D=AB$. Since $A^2=B^2=C^2$ by Proposition \ref{masrex},
we have  $D^2=ABAB=ABBAC=C$, so $D^4=A^2$, and $D^A=BA=D^{-1}$.
Thus $J$ is an epimorphic image of the generalized quaternion group $Q_{16}=\langle u,v\,|\, u^4=v^2, u^v=u^{-1}\rangle$
of order 16. As $|J|=16$, we have $J\cong Q_{16}$. As the order of $D$ is 8, this is the exponent of $J$.
\end{proof}

\section{Exponent of $J$}\label{secv8}

\begin{prop}\label{z3} Suppose we are in Case 1. Then  $Z(J)=J^{p^{2m}}$, $Z_3(J)=J^{p^{m}}$ is abelian of order $p^{4m}$ with presentation
$$
\langle X,Y,Z\,|\,XY=YX,\; YZ=ZY,\; XZ=ZX,\; X^{p^{2m}}=1, X^{p^{m}}Y^{p^{m}}=1,\;  Z^{p^{m}}=1\rangle,
$$
$J/Z_3(J)\cong\mathrm{Heis}(\Z/p^{m}\Z)$, and $J$ has exponent $p^{3m}$.
\end{prop}

\begin{proof} That $Z_3(J)=J^{p^{m}}$ was established in the proof of Theorem \ref{main2}. Referring to
the group $\langle X_0,Y_0\rangle$ constructed in Theorem \ref{model}, we have an
isomorphism $J\to\langle X_0,Y_0\rangle$, defined by $A\mapsto X_0$, $B\mapsto Y_0$. Then 
$Z_3(J)=\langle  A^{p^{m}}, B^{p^{m}}, C^{p^{m}}\rangle$ corresponds to 
$\langle  X, Y, Z\rangle$, which gives the desired order and presentation of $Z_3(J)$. As $Z_3(J)$
is abelian, $A^{p^{2m}}B^{p^{2m}}=1$, and $C^{p^{2m}}=1$,  we infer $J^{p^{2m}}=Z(J)$.
The fact that 
$J/Z_3(J)\cong\mathrm{Heis}(\Z/p^{m}\Z)$ was demonstrated in the proof of Theorem \ref{main2}. Thus, 
there is an epimorphism $f:J\to \mathrm{Heis}(\Z/p^{m}\Z)$, where $\ker(f)=Z_3(J)=J^{p^m}$ has exponent $p^{2m}$.
Let $g\in J$. Then $g^{p^{m}}\in\ker(f)$, so $g^{p^{3m}}=1$. Since $A$ has order $p^{3m}$ by Proposition \ref{orden}, 
the result follows.
\end{proof}  

\begin{prop}\label{z3p2} Suppose we are in Case 2 and $m>1$. Then $Z_3(J)$ is abelian of order $2^{4m-2}$, with presentation
$$
\langle x,y,z\,|\, xy=yx, xz=zx, yz=zy, x^{2^{2m-2}}=z^{2^m}, x^{2^{m-1}}y^{2^{m-1}}=1, x^{2^{2m-1}}=1\rangle,
$$
$J/Z_3(J)\cong \mathrm{Heis}(\Z/2^{m}\Z)/U$, where $U$ is the $2^{m-1}$th power of the center of $\mathrm{Heis}(\Z/2^{m}\Z)$,
and $J$ has exponent $2^{3m-1}$.
\end{prop}

\begin{proof} Referring to
the group $\langle x_0,y_0\rangle$ constructed in the proof of Theorem \ref{model2}, we have an
isomorphism $J\to\langle x_0,y_0\rangle$ such that $A\mapsto x_0$ and $B\mapsto y_0$. Then 
$Z_3(J)=\langle  A^{2^m}, B^{2^m}, C^{2^{m-1}}\rangle$ corresponds to 
$\langle  x, y, z\rangle$, which gives the desired order and presentation of $Z_3(J)$. 
The fact that 
$J/Z_3(J)\cong\mathrm{Heis}(\Z/2^m\Z)/U$ was demonstrated in the proof of Theorem \ref{main2}. Thus, 
there is an epimorphism $f:J\to \mathrm{Heis}(\Z/2^{m}\Z)/U$, where $\ker(f)=Z_3(J)$ has exponent $2^{2m-1}$.
A matrix calculation shows that $\mathrm{Heis}(\Z/2^{m}\Z)$ has exponent $2^{m+1}$ and that $\mathrm{Heis}(\Z/2^{m}\Z)/U$
has exponent~$2^m$. Let $g\in J$. Then $g^{2^{m}}\in\ker(f)$, so $g^{2^{3m-1}}=1$. Since $A$ has order $2^{3m-1}$ 
by Proposition \ref{orden}, the result follows.
\end{proof}

\begin{prop}\label{z3p3} Suppose we are in Case 3. Then

(a) We have $Z_5(J)=\langle A^3,B^3,C^3\rangle=J^3$, where $J/Z_5(J)\cong\mathrm{Heis}(\Z/3\Z)$.

(b) The group $Z_3(J)=\langle A^9,B^9,C^9\rangle=J^9$ is abelian of exponent 9, and $Z_5(J)/Z_3(J)\cong (\Z/3\Z)^3$.

(c) We have $Z(J)=\langle A^{27}\rangle=J^{27}$ and the exponent of $J$ is equal to $81$.
\end{prop}

\begin{proof} (a) This was demonstrated in the proof of Theorem \ref{main2}.

(b) We already know that $Z_3(J)=\langle A^9,B^9,C^9\rangle$ from Theorem \ref{main2}. Clearly
$\langle A^9,B^9,C^9\rangle\subseteq J^9$. Referring to the group $\langle x_0,y_1\rangle$ constructed in the proof of Theorem \ref{model3}, we have an isomorphism $J\to\langle x_0,y_1\rangle$ such that $A\mapsto x_0$ and $B\mapsto y_1$. Then 
$Z_3(J)$ corresponds to $\langle  x^3, y, z^3\rangle$ and $Z_5(J)$ corresponds to $\langle  x, y_0, z\rangle$.
The stated defining relations of $\langle  x, y, z\rangle$ show that $\langle  x, y, z^3\rangle$ is abelian
and its subgroup $\langle  x^3, y, z^3\rangle$ has exponent 9, while the stated defining relations
of $\langle  x, y_0, z\rangle$ prove that $Z_5(J)/Z_3(J)$ is abelian of exponent 3
and order 27, so $Z_5(J)/Z_3(J)\cong (\Z/3\Z)^3$. This and part (a) yield $J^9\subseteq \langle A^9,B^9,C^9\rangle$.

(c) That  $Z(J)=J^{27}$ follows from $A^{27}B^{27}=1$ and parts (a) and (b). As $A$ has order 81, this is the exponent of $J$.
\end{proof}

\section{Order, nilpotency class, and exponent of $G(\alpha)$}\label{secv9}

Recall that $\alpha$ is 3-admissible if one the following three possibilities occur:
$3\nmid(\alpha-1)$; the multiplicity of 3 as a factor of $\alpha-1$ is larger than 1;
$3|(\alpha-1)$ and $(\alpha-1)/3\equiv 1\mod 3$. 

\begin{theorem}\label{todojunto} The nilpotency class of 
the Macdonald group $G(\alpha)$ is equal to 
3 if $\alpha\in\{-1,3\}$; 5 if $\alpha\notin\{-1,3\}$ and $\alpha$ is 3-admissible; and
7 if  $3|(\alpha-1)$ and $(\alpha-1)/3\equiv -1\mod 3$. Moreover, the order and exponent of 
$G(\alpha)$ are respectively equal to $|\alpha-1|^7$ and $|\alpha-1|^3$ if $2\nmid (\alpha-1)$
and $\alpha$ is 3-admissible; $|\alpha-1|^7/8$ and $|\alpha-1|^3/2$ if $v_2(\alpha-1)>1$
and $\alpha$ is 3-admissible;  $|\alpha-1|^7/8$ and $|\alpha-1|^3$ if $v_2(\alpha-1)=1$
and $\alpha$ is 3-admissible; $27|\alpha-1|^7$ and $3|\alpha-1|^3$ if $2\nmid(\alpha-1)$
and $\alpha$ is not 3-admissible; $27|\alpha-1|^7/8$ and $3|\alpha-1|^3/2$ if $v_2(\alpha-1)>1$
and $\alpha$ is not 3-admissible; and $27|\alpha-1|^7/8$ and $3|\alpha-1|^3$ if $v_2(\alpha-1)=1$
and $\alpha$ is not 3-admissible.
\end{theorem}

\begin{proof} By Theorem \ref{nul}, $G(\alpha)$ is the direct product of its Sylow subgroups, whose orders, nilpotency classes,
and exponents are given in Theorems \ref{main1} and \ref{main2} and Propositions \ref{mau}, \ref{z3}, 
\ref{z3p2}, and \ref{z3p3}.
\end{proof}

\section{Appendix}

By a {\em model} of $J$ we understand a
group $T$ that is an image of $J$ and whose order attains the upper bound stated in Proposition \ref{masrex},
that is, $|T|=p^{7m}$ in Case 1, $|T|=2^{7m-3}$ in Case 2, and $|T|=3^{10}$ in Case 3.

The following well-known gadget (cf. \cite[Chapter III, Section 7]{Z}) will be used repeatedly and implicitly to construct a model of 
$J$.

\begin{theorem} Let $T$ be an arbitrary group and $L$ a cyclic group of finite order $n\in\N$. Suppose that $t\in T$
and that $\Omega$ is an automorphism of $T$ fixing $t$ and such that $\Omega^n$ is conjugation by $t$. Then there
is a group $E$ containing $T$ as a normal subgroup, such that $E/T\cong L$, and for some $g\in E$ of order $n$ modulo $T$, we have 
$g^n=t$ and $\Omega$ is conjugation by $g$.
\end{theorem}

Suppose first we are in Case 1. It turns out that 
$Z_3(J)=\langle  A^{p^{m}}, B^{p^{m}}, C^{p^{m}}\rangle$, a normal abelian subgroup of $J$ 
of order $p^{4m}$. The defining relations of $J$ allow us to see how $A$, $B$, and $C$
conjugate the given generators of $Z_3(J)$. This prompts the construction of $J$ below.

\begin{theorem}\label{model} Suppose we are in Case 1. Then there is a model of $J$.
\end{theorem}

\begin{proof} We have $\alpha=1+k p^m$, with $k\in\N$. By adding $p^{3m}$ to $\alpha$, if necessary,
we may assume that $k$ is even. 

We start with an abelian group $\langle X,Y,Z\rangle$ of order $p^{4m}$ and defining
relations
$$
XY=YX,\; YZ=ZY,\; XZ=ZX,\; X^{p^{2m}}=1, X^{p^{m}}Y^{p^{m}}=1,\;  Z^{p^{m}}=1.
$$

We next construct a cyclic extension $\langle X,Y,Z_0\rangle$ 
of $\langle X,Y,Z\rangle$ of order $p^{5m}$, where
$Z_0^{p^{m}}=Z$, by means of an automorphism $\Omega$ 
of $\langle X,Y,Z\rangle$ that fixes $Z$ and such that
$\Omega^{p^{m}}$ is conjugation by~$Z$, that is, the trivial automorphism. In order to achieve this goal,
we consider the assignment
$$
X\mapsto X^{1+ k p^m},\; Y\mapsto Y^{1- k p^m},\; Z\mapsto Z,
$$
where $\beta=1-k p^m$ is the inverse of $\alpha$ modulo $p^{2m}$. We easily verify that the defining relations
of $\langle X,Y,Z\rangle$ are preserved. Thus
the above assignment extends to an endomorphism $\Omega$ of $\langle X,Y,Z\rangle$, which is clearly surjective
and hence an automorphism.

We have $\alpha^{p^{m}}\equiv 1\mod p^{2m}$ by Lemma \ref{valdif}, so $\beta^{p^{m}}\equiv 1\mod p^{2m}$, whence
$\Omega^{p^{m}}$ is the trivial automorphism of $\langle X,Y,Z\rangle$. This produces the required extension, 
where $\Omega$ is conjugation by $Z_0$. We readily verify that $\langle X,Y,Z_0\rangle$
has defining relations:
$$
X^{Z_0}=X^\alpha,\; {}^{Z_0} Y=Y^\alpha,\; XY=YX,\; X^{p^{2m}}=1,\; X^{p^{m}}Y^{p^{m}}=1,
Z_0^{p^{2m}}=1.
$$

We next construct a cyclic extension $\langle X_0,Y,Z_0\rangle$ 
of $\langle X,Y,Z_0\rangle$ of order $p^{6m}$ with
$X_0^{p^{m}}=X$, by means of an automorphism $\Psi$ 
of $\langle X,Y,Z_0\rangle$ that fixes $X$ and such that
$\Psi^{p^{m}}$ is conjugation by~$X$. For this purpose, recalling that $k$ is even, we set
$$
c=\begin{cases} 0 & \text{ if }p\neq 3,\\ k^2 3^{2m-1} & \text{ if }p=3,\end{cases}\quad 
b=\begin{cases} p^m k/2 & \text{ if }p\neq 3,\\ 3^m k/2+k^2 3^{2m-1} & \text{ if }p=3,\end{cases}
$$
so that $b=p^m k/2+c$, and consider the assignment
$$
X\mapsto X,\; Y\mapsto Z_0^{-p^m}\, Y^{1+b}=Z^{-1} Y^{1+b},\; Z_0\mapsto Z_0 X^{-k}.
$$

Let us verify that the defining relations of $\langle X,Y,Z_0\rangle$ are preserved. This is obvious for the first 
and fourth relations. As for the sixth, we need to see that $(Z_0 X^{-k})^{p^{2m}}=1$. This holds because
$$
(Z_0 X^{-k})^{p^{2m}}=Z_0^{p^{2m}} X^{-k(\alpha^{p^{2m}}-1)/(\alpha-1)}=1
$$
by Lemma \ref{valdif}. The second, third and fifth relations are easily seen to be preserved, using that $Z_0^{-p^{m}}=Z^{-1}$ commutes with $X$ and $Y$, as well as $Z^{p^m}=1=Y^{p^{2m}}$.
Thus the above assignment extends to an endomorphism $\Psi$ of $\langle X,Y,Z_0\rangle$. As $\gcd(1+b,p^{2m})=1$,  $\Psi$
is an automorphism. We next show that $\Psi^{p^{m}}$ is 
conjugation by $X$. By definition, $\Psi^{p^{m}}$ fixes $X$. Moreover, 
$$Z_0\Psi^{p^{m}}=Z_0X^{-k p^{m}}=Z_0X^{1-\alpha},\; Z_0^X=Z_0 Z_0^{-1}X^{-1}Z_0 X=Z_0X^{-\alpha} X=Z_0X^{1-\alpha}.
$$
It remains to verify that $\Psi^{p^{m}}$ fixes $Y$. By definition, we have 
$$
Y\Psi=Z^{-1} Y Y^b,
$$
which implies that $\Psi$ fixes $Y^{\alpha-1}$ and $Y^b$. On the other hand, use of Lemma \ref{valdif} yields
$$
Z\Psi=Z_0^{p^m}\Psi=(Z_0\Psi)^{p^m}=(Z_0 X^{-k})^{p^m}=Z_0^{p^m}
X^{-k(\alpha^{p^m}-1)/(\alpha-1)}=ZX^{-k p^m}=Z X^{1-\alpha}=Z Y^{\alpha-1}.
$$
Thus
$$
Y\Psi^2=(Z^{-1}YY^b)\Psi=Z^{-1}Y^{1-\alpha}Z^{-1} Y Y^b Y^b=Z^{-2}Y Y^{1-\alpha} Y^{2b},
$$
$$
Y\Psi^3=(Z^{-2}Y Y^{1-\alpha} Y^{2b})\Psi=Z^{-2}Y^{2(1-\alpha)}Z^{-1} Y Y^b Y^{1-\alpha} Y^{2b}=Z^{-3}Y Y^{3(1-\alpha)} Y^{3b},
$$
and in general
$$
Y\Psi^i=Z^{-i}Y Y^{{{i}\choose{2}}(1-\alpha)} Y^{ib},\quad i\geq 2.
$$
In particular,
$$
Y\Psi^{p^m}=Z^{-p^m}Y Y^{{{p^m}\choose{2}}(1-\alpha)} Y^{p^m b}=Y.
$$
This produces the required extension, where $\Psi$ is conjugation by $X_0$. We readily verify that $\langle X_0,Y,Z_0\rangle$ has defining relations 
$$
X_0^{p^{3m}}=1,\; X_0^{Z_0}=X_0^\alpha,\; 
 Y^{X_0}=Z_0^{-p^m}\, Y^{1+b},\;
{}^{Z_0} Y=Y^\alpha,
$$
$$
X_0^{p^{m}}Y=YX_0^{p^{m}},\;
X_0^{p^{2m}}Y^{p^m}=1,\;
Z_0^{p^{2m}}=1.
$$
Here $X_0^{Z_0}=X_0^\alpha$ is equivalent to $Z_0^{X_0}=Z_0 X^{-k}=Z_0 X_0^{-k p^{m}}=
Z_0 X_0^{1-\alpha}$.

We finally construct a cyclic extension $\langle X_0,Y_0,Z_0\rangle$ 
of $\langle X_0,Y,Z_0\rangle$ of order $p^{7m}$ with
$Y_0^{p^{m}}=Y$, by means of an automorphism $\Pi$
of $\langle X_0,Y,Z_0\rangle$ that fixes $Y$ and such that
$\Pi^{p^{m}}$ is conjugation by~$Y$. With this aim in mind, we consider the assignment 
$$
X_0\mapsto X_0Z_0,\; Y\mapsto Y,\; Z_0\mapsto Y^{k} Z_0. 
$$
Let us verify that the defining relations of $\langle X_0,Y,Z_0\rangle$ are preserved. Regarding the first relation,
$$
(X_0Z_0)^{p^{3m}}=Z_0^{p^{3m}} X_0^{\alpha (\alpha^{p^{3m}}-1)/(\alpha-1)}=1
$$
by Lemma \ref{valdif}. Likewise, the seventh relation is preserved, as Lemma \ref{valdif} ensures
$$
(Y^{k}Z_0)^{p^{2m}}= Y^{k(\alpha^{p^{2m}}-1)/(\alpha-1)} Z_0^{p^{2m}} =Y^{k p^{2m}} Z_0^{p^{2m}}=1.
$$
The preservation of the fourth relation is obvious. Regarding the fifth relation, Lemma \ref{valdif} yields
$$
(X_0 Z_0)^{p^{m}}=Z_0^{p^{m}} X_0^{\alpha (\alpha^{p^{m}}-1)/(\alpha-1)}=Z X^\ell,\quad \ell\in\Z,
$$
where both factors commute with $Y$.
As for the sixth relation, observe that
$$
(X_0 Z_0)^{p^{2m}}=Z_0^{p^{2m}} X_0^{\alpha (\alpha^{p^{2m}}-1)/(\alpha-1)}=X_0^{p^{2m}}
$$
by Lemma \ref{valdif}, so $(X_0 Z_0)^{p^{2m}}Y^{p^m}=X_0^{p^{2m}}Y^{p^m}=1$. In regards to the third relation, we must prove 
$$
Y^{X_0 Z_0}=(Y^{k}Z_0)^{-p^m} Y Y^b.
$$
Note that $(Y^b)^{Z_0}=Y^b$, so
$$
Y^{X_0 Z_0}=(Z_0^{-p^m} Y Y^b)^{Z_0}=Z_0^{-p^m} Y^{Z_0}\, Y^b=Z_0^{-p^m} Y^{1-k p^m}\, Y^b.
$$
On the other hand,  Lemma \ref{valdif} ensures that
$$
(Y^k Z_0)^{p^m}=Y^{k (\alpha^{p^m}-1)/(\alpha-1)}Z_0^{p^m}=
Y^{k p^m}Z_0^{p^m},
$$
so
$$
(Y^{k}Z_0)^{-p^m} Y Y^b=Z_0^{-p^m}Y^{-k p^m}Y\, Y^b=Z_0^{-p^m} Y^{1-k p^m}\, Y^b,
$$
as required. The preservation of the second relation requires more work. We must show that
\begin{equation}
\label{mus}
(X_0 Z_0)^{Y^{k} Z_0}=(X_0 Z_0)^{\alpha}.
\end{equation}
We begin by obtaining a formula for $X_0^{Y^{k}}$. From $X_0^{-1}Y X_0=Y^{X_0}=Z^{-1} Y Y^b$, we deduce
$Y^{-1}X_0^{-1} Y=Z^{-1} Y^b X_0^{-1}$, and therefore 
\begin{equation}
\label{antesqw}
X_0^Y=X_0 Y^{-b} Z.
\end{equation}
Thus
\begin{equation}
\label{antes}
X_0^{Y^{k}}=X_0 Y^{-bk} Z^k=X_0 Y^{-bk} Z_0^{k p^m}.
\end{equation}
We next obtain a formula for $Z_0^{Y^{k}}$. From $Z_0Y Z_0^{-1}={}^{Z_0}Y=Y^\alpha$, we infer
$Y^{-1}Z_0 Y=Y^{\alpha-1}Z_0$. Noting that $[Z_0,Y^{\alpha-1}]=1$, we deduce $Z_0^Y=Z_0 Y^{\alpha-1}$,
and hence 
\begin{equation}
\label{pantes}
Z_0^{Y^k}=Z_0 Y^{k(\alpha-1)}.
\end{equation}
Using (\ref{antes}), (\ref{pantes}), and $[Z_0,Y^{\alpha-1}]=1$, we obtain
$$
(X_0 Z_0)^{Y^k}=X_0 Y^{-bk} Z_0^{k p^m}Z_0Y^{k(\alpha-1)}=X_0 Y^{k(-b+(\alpha-1))} Z_0^{\alpha}.
$$
As $Z_0$ commutes with $Y^b$ and $Y^{\alpha-1}$, we infer
\begin{equation}
\label{artes1}
(X_0 Z_0)^{Y^k Z_0}=X_0^\alpha Y^{k(-b+(\alpha-1))} Z_0^{\alpha},
\end{equation}

On the other hand, we have
\begin{equation}
\label{artes2}
(X_0 Z_0)^{\alpha}=Z_0^{\alpha} X_0^{\alpha (\alpha^{\alpha}-1)/(\alpha-1)}.
\end{equation}

Taking into account (\ref{artes1}) and (\ref{artes2}), as well as the fact that $Z_0$
commutes with $Y^b$ and $Y^{\alpha-1}$, we see that (\ref{mus}) is equivalent to
\begin{equation}
\label{artes3}
X_0^{\alpha (\alpha^{\alpha}-1)/(\alpha-1)}=X_0^{\alpha\alpha^{\alpha}}\, Y^{k(-b+(\alpha-1))}.
\end{equation}
Making use of the fundamental relation $X_0^{p^{2m}}Y^{p^m}=1$ and the meanings of $b$ and $c$, we find that
$$
Y^{k(\alpha-1)}=X_0^{-(\alpha-1)^2},\; Y^{-kb}=X_0^{(\alpha-1)^2/2+kc3^m}.
$$
Thus (\ref{artes3}) is equivalent to
\begin{equation}
\label{artes4}
X_0^{\alpha (\alpha^{\alpha}-1)/(\alpha-1)}=X_0^{\alpha\alpha^{\alpha}-(\alpha-1)^2/2+kc 3^m}.
\end{equation}
Here
\begin{equation}
\label{artes5}
\alpha^{\alpha}=(1+(\alpha-1))^{\alpha}=1+\alpha(\alpha-1)+{{\alpha}\choose{2}}(\alpha-1)^2+{{\alpha}\choose{3}}(\alpha-1)^3+
{{\alpha}\choose{4}}(\alpha-1)^4+\cdots
\end{equation}
We see from (\ref{artes5}) that $\alpha^\alpha\equiv 1+\alpha(\alpha-1)\mod p^{3m}$, so
\begin{equation}
\label{artes6}
\alpha\alpha^\alpha\equiv \alpha+\alpha^2(\alpha-1)\mod p^{3m}.
\end{equation}
We also derive from (\ref{artes5}) that 
$$
(\alpha^{\alpha}-1)/(\alpha-1)\equiv \alpha+{{\alpha}\choose{2}}(\alpha-1)+{{\alpha}\choose{3}}(\alpha-1)^2\mod p^{3m},
$$
and since $\alpha p^{2m}\equiv p^{2m}\mod p^{3m}$, we infer
\begin{equation}
\label{artes7}
\alpha (\alpha^{\alpha}-1)/(\alpha-1)\equiv \alpha^2+(\alpha-1)^2/2+{{\alpha}\choose{3}}(\alpha-1)^2\mod p^{3m}.
\end{equation}
From (\ref{artes6}) and (\ref{artes7}) we see that (\ref{artes4}) follows from
$$
\alpha^2+(\alpha-1)^2/2+{{\alpha}\choose{3}}(\alpha-1)^2\equiv \alpha+\alpha^2(\alpha-1)-(\alpha-1)^2/2+kc 3^m
\mod p^{3m},
$$
which is equivalent to
\begin{equation}
\label{artes8}
\alpha^2+(\alpha-1)^2+{{\alpha}\choose{3}}(\alpha-1)^2\equiv \alpha+\alpha^2(\alpha-1)+kc 3^m
\mod p^{3m}.
\end{equation}
Here
$$
\alpha^2+(\alpha-1)^2\equiv 1+2(\alpha-1)+2(\alpha-1)^2\equiv \alpha+\alpha^2(\alpha-1)
\mod p^{3m},
$$
so (\ref{artes8}) means
$$
{{\alpha}\choose{3}}(\alpha-1)^2\equiv kc 3^m\mod p^{3m},
$$
which is readily seen to be true whether $p\neq 3$ or $p=3$. This
completes the verification that the second defining relation of $\langle  X_0,Y,Z_0\rangle$ is preserved, which
ensures the existence of an automorphism~$\Pi$
of $\langle X_0,Y,Z_0\rangle$ fixing $Y$ and extending the given assignment. We must now verify that $\Pi^{p^m}$
is conjugation by~$Y$. By definition, $Y\Pi^{p^m}=Y$ and 
$$
Z_0\Pi^{p^m}=Y^{k p^m}Z_0=Y^{\alpha-1}Z_0=Z_0^Y.
$$
It remains to show that $X_0\Pi^{p^m}=X_0^Y$. The calculation of $X_0^Y$ is achieved in (\ref{antesqw}).
On the other hand, repeated application of $X_0\Pi=X_0 Z_0$, 
$Y=Y\Pi$, $Z_0\Pi=Y^{k} Z_0$ yields
$$
X_0\Pi^i=X_0 Y^{k\alpha(1+2\alpha+3\alpha^2+\cdots+(i-1)\alpha^{i-2})}Z_0^i,\quad i\geq 2.
$$
Set $U=k\alpha(1+2\alpha+3\alpha^2+\cdots+(p^m-1)\alpha^{p^m-2})$. We are thus reduced to show that
$$
U\equiv -b\mod p^{2m}.
$$
Now
$$
\alpha^i=(1+(\alpha-1))^i\equiv 1+i(\alpha-1)\mod   p^{2m},\quad i\geq 0,
$$
so 
$$
U\equiv k\alpha(1+2+\cdots+(p^m-1)+(\alpha-1)(1\times 2+2\times 3+\cdots+
(p^m-2)\times (p^m-1)))\mod p^{2m}.
$$
As is well known, we have
$$
\underset{1\leq i\leq n}\sum i=\frac{n(n+1)}{2},\; \underset{1\leq i\leq n}\sum i^2=\frac{n(n+1)(2n+1)}{6},\quad n\geq 1,
$$
which implies
$$
\underset{1\leq i\leq n-2}\sum i(i+1)=\frac{(n-2)(n-1)n}{3},\quad n\geq 3.
$$
Set $n=p^m$. We then have
$$
k\alpha \frac{(n-1)n}{2}\equiv \frac{k\alpha (p^m-1)p^m}{2}\equiv \frac{-k p^m}{2}\mod p^{2m},
$$
so we are reduced to show that
$$
k\alpha(\alpha-1)\frac{(n-2)(n-1)n}{3}\equiv -c\mod p^{2m},
$$
which is readily seen to be true whether $p\neq 3$ or $p=3$. This produces the required extension, where $\Pi$ is conjugation by $Y_0$. From $X_0^{Y_0}=X_0Z_0$, we deduce $[X_0,Y_0]=Z_0$, whence $\langle X_0,Y_0,Z_0\rangle=\langle X_0,Y_0\rangle$.
Moreover, we have $X_0^{Z_0}=X_0^\alpha$ and $Y_0^{-1}Z_0Y_0=Z_0^{Y_0}=Y^{k}Z_0=
Y_0^{k p^m}Z_0=Y_0^{\alpha-1}Z_0$, which implies
${}^{Z_0} Y_0=Y^\alpha$. Finally, we also have $X_0^{p^{3m}}=1=Y_0^{p^{3m}}$.
\end{proof}

\begin{note}
The attentive reader will notice that if $p=3$ the restriction $m>1$ or $(\alpha-1)/3\equiv 1\mod 3$
was never used in the proof of Theorem \ref{model}. Thus, if  $m=1$ and $(\alpha-1)/3\equiv -1\mod 3$,
Theorem \ref{model} still constructs an image of $J$ of order $3^7$. However,
$|J|=3^{10}$ in this case, as seen in Theorem \ref{model3} below.
\end{note}

Suppose next we are in Case 2. It transpires that
$\langle  A^{2^{m}}, B^{2^{m}}, C^{2^{m-1}}\rangle$, a normal abelian subgroup of $J$ of
order $2^{4m-2}$, is equal to $Z_3(J)$ if $m>1$ and to $Z_2(J)$ if $m=1$.
Using the defining relations of~$J$, we can determine the precise way in which $A$, $B$, and $C$
conjugate the given generators of~$Z_3(J)$, which suggest the following construction of $J$.

\begin{theorem}\label{model2} Suppose we are in Case 2. Then there is a model of $J$.
\end{theorem}

\begin{proof} We have $\alpha = 1 + 2^mk$, with $k\in\N$ odd. 
We start with an abelian group $\langle x,y,z\rangle$ of order $2^{4m-2}$ generated by elements $x,y,z$ subject to the
defining relations:
    \[
    xy=yx,\;
    xz=zx,\;
    yz=zx,\;
    z^{2^m}=x^{2^{2m-2}},\;
    x^{2^{m-1}}y^{2^{m-1}}=1,\;
    x^{2^{2m-1}}=1.
    \]
    
		We next construct a cyclic extension $\langle x,y,z_0\rangle$ 
of $\langle x,y,z\rangle$ of order $2^{5m-3}$, where
$z_0^{2^{m-1}}=z$, by means of an automorphism $\Omega$ 
of $\langle x,y,z\rangle$ that fixes $z$ and such that
$\Omega^{2^{m-1}}$ is conjugation by~$z$, that is, the trivial automorphism. In order to achieve this goal,
we consider the assignment
		\[
    x\mapsto x^\alpha,\;
    y\mapsto y^\beta,\;
    z\mapsto z,
    \]
		where $\beta=1-2^mk$ is the inverse of $\alpha$ modulo $2^{2m}$. The defining relations of $\langle x,y,z\rangle$ are easily seen to be preserved. Thus the above assignment extends to an endomorphism $\Omega$ of $\langle x,y,z\rangle$ which is clearly surjective
		and hence an automorphism of $\langle x,y,z\rangle$. Let us verify that $\Omega^{2^{m-1}}$ acts trivially on $x,y,z$.
		This is obviously true for $z$, and since $\alpha^{2^{m-1}}\equiv 1 \mod 2^{2m-1}$ and $\beta^{2^{m-1}}\equiv 1 \mod 2^{2m-1}$,
		it is also true of $x$ and $y$. This produces the required extension, where $\Omega$ is conjugation by $z_0$. We readily verify that $\langle x,y,z_0\rangle$ has defining relations:
    \[
		xy = yx,\;
    x^{z_0} = x^\alpha,\;
    {}^{z_0}y=y^\alpha,\;
    z_0^{2^{2m-1}} = x^{2^{2m-2}},\;
		x^{2^{m-1}} y^{2^{m-1}} = 1,\;
    x^{2^{2m-1}} = 1.
    \]
    
		We next construct a cyclic extension $\langle x_0,y,z_0\rangle$ 
of $\langle x,y,z_0\rangle$ of order $2^{6m-3}$, where
$x_0^{2^{m}}=x$, by means of an automorphism $\Psi$ 
of $\langle x,y,z_0\rangle$ that fixes $x$ and such that
$\Psi^{2^{m}}$ is conjugation by~$x$. For this purpose,
we consider the assignment
	 \[
    x\mapsto x,\;
    y\mapsto z_0^{-2^m}y^{1+2^{m-1}k}=z^{-2} y^{1+2^{m-1}k},\;
    z_0\mapsto z_0x^{-k}.
    \]
    Let us verify that the defining relations of $\langle x,y,z_0\rangle$ are preserved. This is easily seen to be true for
	the first, second, and sixth relations. As for the fifth relation, since $k$ is odd, we have
	$$
	(z^{-2} y^{1+2^{m-1}k})^{2^{m-1}}=z^{-2^m}y^{2^{2m-2}k}y^{2^{m-1}}=x^{-2^{2m-2}}y^{2^{2m-2}}y^{2^{m-1}}=
	x^{-2^{2m-2}}x^{-2^{2m-2}}y^{2^{m-1}}=y^{2^{m-1}},
	$$
	as required. Regarding the fourth relation, we have
	\[
    (z_0x^{-k})^{2^{2m-1}}
    = z_0^{2^{2m-1}} x^{-k(1 + \alpha + \cdots + \alpha^{2^{2m-1}-1})}=z_0^{2^{2m-1}}=x^{2^{2m-2}},
    \]
	since $(\alpha^{2^{2m-1}}-1)/(\alpha-1)\equiv  0\mod 2^{2m-1}$. In regards to third relation, we have
    \[
    {}^{(z_0x^{-k})}(z^{-2}y^{1+2^{m-1}k})= z^{-2} y^{\alpha (1+2^{m-1}k)}=z^{-2\alpha} y^{\alpha (1+2^{m-1}k)}=
		(z^{-2}y^{1+2^{m-1}k})^\alpha,
    \]
		as $2\alpha\equiv 2\mod 2^{m+1}$. Thus the above assignment extends to an endomorphism $\Psi$ of $\langle x,y,z_0\rangle$.
		Since $x,y^{1+2^{m-1}k},z_0\in\mathrm{im}(\Psi)$ and $x^{2^{m-1}}y^{2^{m-1}}=1$, it follows that $\Psi$ is 
		surjective and hence an automorphism of $\langle x,y,z_0\rangle$. 
		
		Let us verify that $\Psi^{2^{m}}$ 
		acts via conjugation by $x$ on $x,y,z$. This is obviously true for $x$. As for~$z_0$, from $x^{z_0}=x^\alpha$ we derive
		$z_0^x=z_0x^{1-\alpha}=z_0 x^{-2^m k}=z_0^{\Psi^{2^m}}$. Regarding $y$, carefully using the defining relations of $\langle x,y,z_0\rangle$ we see by induction that 
		$y^{\Psi^n} = z^{-2n} y^{1+n(2-n)2^{m-1}k}$ for all $n\in\N$. In particular,
    \[
    y^{\Psi^{2^m}}
    = z^{-2(2^m)} y y^{(2-2^m)2^{2m-1}k}
    = y
    = y^x.
    \]
    
		This produces the required extension, where $\Psi$ is conjugation by $x_0$. We readily verify that $\langle x_0,y,z_0\rangle$ has defining relations:
    \[
    y^{x_0} = z_0^{-2^m}y^{1+2^{m-1}k},\;
    x_0^{z_0} = x_0^\alpha,\;
    {}^{z_0}y = y^\alpha,\;
    z_0^{2^{2m-1}} = x_0^{2^{3m-2}},\;
		x_0^{2^{2m-1}}y^{2^{m-1}}=1,\;
		x_0^{2^{3m-1}} = 1,
    \]
		where $x_0^{z_0} = x_0^\alpha$ is equivalent to the given relation $z_0^{x_0}=z_0x^{-k}=z_0x_0^{1-\alpha}$.
		
		We finally construct a cyclic extension $\langle x_0,y_0,z_0\rangle$ 
of $\langle x_0,y,z_0\rangle$ of order $2^{7m-3}$, where
$y_0^{2^{m}}=y$, by means of an automorphism $\Pi$ 
of $\langle x_0,y,z_0\rangle$ that fixes $y$ and such that
$\Pi^{2^{m}}$ is conjugation by~$y$. For this purpose,
we consider the assignment
    \[
    x_0\mapsto x_0z_0,\;
    y\mapsto y,\;
    z_0\mapsto y^kz_0.
    \]
    
    Let us verify that the defining relations of $\langle x_0,y,z_0\rangle$ are preserved. This is obviously true for the third
		relation. As for the first relation, we have
    \[
    y^{x_0z_0}
    = (z_0^{-2^m} y^{1+2^{m-1}k})^{z_0}
    = z_0^{-2^m} y^{\beta (1+2^{m-1}k)}
    = z_0^{-2^m} y^{\beta + 2^{m-1}k},
    \]
    since $\beta 2^{m-1}k\equiv 2^{m-1}k\mod 2^{2m-1}$. On the other hand,
    \[
    (y^kz_0)^{-2^m} y^{1+2^{m-1}k}
    = (y^{k(1+\alpha +\cdots +\alpha^{2^m-1})}z_0^{2^m})^{-1} y^{1+2^{m-1}k},
    \] 
		where $(\alpha^{2^m}-1)/(\alpha-1)\equiv 2^m\mod 2^{2m-1}$, so
    \[
    (y^kz_0)^{-2^m} y^{1+2^{m-1}k}
    = (y^{2^mk}z_0^{2^m})^{-1} y^{1+2^{m-1}k}
    = z_0^{-2^m} y^{-2^mk} y^{1+2^{m-1}k}
    = z_0^{-2^m} y^{\beta + 2^{m-1}k},
    \]
    as required. Regarding the sixth relation, we have
    \[
    (x_0z_0)^{2^{3m-1}}
    = z_0^{2^{3m-1}} x_0^{\alpha (1 + \alpha + \cdots + \alpha^{2^{3m-1}-1})}
    = z_0^{2^{3m-1}}
    = 1,
    \]
    since $(\alpha^{2^{3m-1}}-1)/(\alpha-1)\equiv 0\mod 2^{3m-1}$. In regards to the fifth relation, we have
    \[
    (x_0z_0)^{2^{2m-1}}
    = z_0^{2^{2m-1}} x_0^{\alpha (1 + \alpha + \cdots + \alpha^{2^{2m-1}-1})} 
    = z_0^{2^{2m-1}} x_0^{2^{2m-1} - 2^{3m-2}}=x_0^{2^{2m-1}},
    \]
    as $\alpha(\alpha^{2^{2m-1}}-1)/(\alpha-1)\equiv 2^{2m-1} - 2^{3m-2}\mod 2^{3m-1}$ by Lemma \ref{valdif2}.
    We next deal with the fourth relation. We have
    \[
    (y^kz_0)^{2^{2m-1}}
    = y^{k(1+\alpha +\cdots +\alpha^{2^{2m-1}-1})}z_0^{2^{2m-1}}
    = z_0^{2^{2m-1}}
    = x_0^{2^{3m-2}},
    \]
		using once again that $(\alpha^{2^{2m-1}}-1)/(\alpha-1)\equiv  0\mod 2^{2m-1}$. On the other hand, we have
    \[
    (x_0z_0)^{2^{3m-2}}
    = z_0^{2^{3m-2}} x_0^{\alpha (1 + \alpha + \cdots + \alpha^{2^{3m-2}-1})}
    = z_0^{2^{3m-2}} x_0^{2^{3m-2} - 2^{4m-3}}
    = x_0^{2^{4m-3}} x_0^{2^{3m-2} - 2^{4m-3}}
    = x_0^{2^{3m-2}}.
    \]
    Indeed, setting $\gamma=\alpha^{2m-1}$, then $\gamma-1\equiv 0\mod 2^{3m-1}$, so 
		$(\gamma^{2^{m-1}}-1)/(\gamma-1)\equiv 2^{m-1}\mod 2^{3m-1}$, and therefore 
		$$
	\alpha\frac{\alpha^{2^{3m-2}}-1}{\alpha-1}\equiv \alpha  \frac{\alpha^{2^{m-1}}-1}{\alpha-1} \frac{\gamma^{2^{m-1}}-1}{\gamma-1}
	\equiv (2^{2m-1} - 2^{3m-2})2^{m-1}\equiv 2^{3m-2} - 2^{4m-3}\mod 2^{3m-1}.
	$$
		
	It remains to verify that the second relation is preserved.	From $z_0yz_0^{-1}=y^\alpha$, we infer
	$y^{-1}z_0 y=y^{\alpha-1}z_0$, so $z_0^y=y^{2^m k} z_0$ and therefore $z_0^{y^k}=y^{2^m k^2} z_0$.
	Moreover, from $x_0^{-1}yx_0=z^{-2}y^{1+2^{m-1}k}$, we deduce
	$y^{-1}x_0^{-1} y=z^{-2}y^{2^{m-1}k}x_0^{-1}$, hence $x_0^{y}=x_0y^{-2^{m-1} k} z^2$, and therefore
	$x_0^{y^k}=x_0y^{-2^{m-1} k^2} z^{2k}$. Thus
    \begin{align*}
        (x_0z_0)^{y^kz_0}
        &=(x_0 y^{-2^{m-1}k^2} z^{2k} y^{2^mk^2}z_0)^{z_0}
        = (x_0 y^{2^{m-1}k^2} z_0^\alpha)^{z_0}
        =x_0^\alpha y^{\beta 2^{m-1}k^2} z_0^\alpha\\
        &= x_0^{\alpha - 2^{2m-1}\beta k^2} z_0^\alpha
        = z_0^\alpha z_0^{-\alpha} x_0^{\alpha - 2^{2m-1}\beta k^2} z_0^\alpha
        = z_0^\alpha x_0^{\alpha^\alpha (\alpha - 2^{2m-1}\beta k^2)}.
    \end{align*} 
    Here
    $$
		\alpha^\alpha\equiv 1+\alpha(\alpha-1)\equiv 1 + 2^mk + 2^{2m}k^2\mod 2^{3m-1},
    $$
    and
    \[
    \alpha - 2^{2m-1}\beta k^2
    \equiv 1 + 2^mk - 2^{2m-1}k^2
    \mod 2^{3m-1},
    \] 
    so
    \[
    \alpha^\alpha (\alpha - 2^{2m-1}\beta k^2)
    \equiv (1 + 2^mk + 2^{2m}k^2)(1 + 2^mk - 2^{2m-1}k^2)
    \equiv 1 + 2^{m+1}k + 3\times 2^{2m-1}k^2\mod 2^{3m-1}.
    \] 
    It follows that
    $$(x_0z_0)^{y^kz_0} = z_0^\alpha x_0^{1 + 2^{m+1}k + 3\times 2^{2m-1}k^2}.$$
    
    On the other hand, we have
    $
    (x_0z_0)^\alpha
    = z_0^\alpha x_0^{\alpha (1 + \alpha + \cdots + \alpha^{\alpha-1})},
    $
    where
		$$
		\frac{\alpha^\alpha-1}{\alpha-1}\equiv \alpha+\alpha(\alpha-1)^2/2\equiv 1 + 2^mk + 2^{2m-1}k^2\mod 2^{3m-1},
    $$
    so
    \[
    \alpha \frac{\alpha^\alpha-1}{\alpha-1}\equiv\alpha (1 + 2^mk + 2^{2m-1}k^2)
    \equiv 1 + 2^{m+1}k + 3\times 2^{2m-1}k^2\mod 2^{3m-1},
    \]
    and therefore
    $$(x_0z_0)^\alpha = z_0^\alpha x_0^{1 + 2^{m+1}k + 3\times 2^{2m-1}k^2},$$
    as required. Thus the given assignment extends to an endomorphism $\Pi$ of $\langle x_0,y,z_0\rangle$, which is
		clearly surjective and hence an automorphism $\Pi$ of $\langle x_0,y,z_0\rangle$. We next verify that $\Pi^{2^m}$ acts
		as conjugation by $y$ on $x_0,y,z_0$. This is obvious for $y$. As for $z_0$, we have 
   $z_0^{\Pi^{2^m}} = y^{2^m k}z_0=z_0^y$, as computed earlier. Regarding $x_0$, we easily see by
   induction that
    $x_0^{\Pi^n}
    = x_0 y^{k(\alpha + 2\alpha^2 + 3\alpha^3 + \cdots + (n-1)\alpha^{n-1})} z_0^n$
    for all $n\in\N$, where the indicated sum has $n-1$ terms and is equal to 0 when $n=1$. Now  
		$\alpha^i\equiv (1+(\alpha-1))^i\equiv 1+i(\alpha-1)\mod 2^{2m-1}$ for all $i\in\N$, so
		$$
		\alpha + 2\alpha^2 + 3\alpha^3 + \cdots + (n-1)\alpha^{n-1}\equiv 1+\cdots +(n-1)+(\alpha-1)(1^2+\cdots +(n-1)^2)\mod 2^{2m-1},
		\quad n\in\N,
		$$
		that is,
		$$
    \alpha + 2\alpha^2 + 3\alpha^3 + \cdots + (n-1)\alpha^{n-1}\equiv \frac{(n-1)n}{2} + (\alpha-1)
		\frac{(n-1)n(2(n-1)+1)}{6},\quad n\in\N.
		$$
		In particular, for $n=2^m$, we have
		\begin{align*}
		\alpha + 2\alpha^2 + 3\alpha^3 + \cdots + (2^m-1)\alpha^{2^m-1} &\equiv 2^{2m-1} - 2^{m-1} + 2^{2m-1}k\frac{(2^m-1)(2^{m+1}-1)}{3}\\
    &\equiv -2^{m-1}\mod 2^{2m-1}.
    \end{align*}
    It follows that
    $x_0^{\Pi^{2^m}}= x_0 y^{-2^{m-1}k} z_0^{2^m} = x_0^y$, as computed earlier.
		
		This produces the required extension, where $\Pi$ is conjugation by $y_0$. We already had $x_0^{z_0}=x_0^\alpha$.
		Moreover, the new relation $z_0^{y_0}=y^k z_0=y_0^{2^m k}z_0=y_0^{\alpha-1}z_0$ is equivalent to ${}^{z_0} y_0=y_0^{\alpha}$.
		Furthermore, from $x_0^{y_0}=x_0 z_0$ we infer $[x_0,y_0]=z_0$, so $\langle x_0,y_0,z_0\rangle=\langle x_0,y_0\rangle$.
		As $x_0^{2^{3m-1}}=1=y_0^{2^{3m-1}}$ the proof is complete.
    \end{proof}

Suppose finally that we are in Case 3.
It turns out that $\langle  A^{3}, B^{9}, C^{3}\rangle$ is a normal subgroup of $J$ of order $3^6$.
The defining relations of $J$ allow us to see how $A$, $B$, and $C$ act on $\langle  A^{3}, B^{9}, C^{3}\rangle$
by conjugation, which prompts the construction of the model of $J$ below.

\begin{theorem}\label{model3} Suppose we are in Case 3. Then there is a model of $J$.
\end{theorem}

\begin{proof} We have $\alpha=1+3k$, where $k\in\N$ and $k\equiv -1\mod 3$.
We start with a group $\langle x,y,z\rangle$ of order $3^6$ generated by elements $x,y,z$ subject to defining relations:
    \[
    x^{27}=1,\;
		xy=yx,\;
    x^9y^3=1,\;
    z^9=1,\;
    x^z=x^{-8},\;
    yz=zy.
    \]
Note that $\langle x,y,z\rangle$ is a semidirect product of $\Z/27\Z\times \Z/3\Z$ by $\Z/9\Z$, where $[z,x^3]=1=[z^3,x]$.

We next construct a cyclic extension $\langle x,y_0,z\rangle$ 
of $\langle x,y,z\rangle$ of order $3^7$, where
$y_0^{3}=y$, by means of an automorphism $\Omega_1$ 
of $\langle x,y,z\rangle$ that fixes $y$ and such that
$\Omega_1^{3}$ is conjugation by~$y$, that is, the trivial automorphism. For this purpose,
we consider the assignment
$$
x\mapsto x^{-8} z^3,\; y\mapsto y,\; z\mapsto y^{-3}z.
$$
All defining relations of $\langle x,y,z\rangle$ are obviously preserved. Thus
the above assignment extends to an endomorphism $\Omega_1$ of $\langle x,y,z\rangle$, which is clearly surjective
and hence an automorphism. We next verify that $\Omega_1^{3}$ agrees with conjugation by $y$ on $x$, $y$, and $z$.
This is obviously true for $y$ and $z$, and noting that $\Omega_1$ fixes $z^3$, we find that it is also true for $x$.
This produces the required extension, 
where $\Omega_1$ is conjugation by $y_0$. We readily verify that $\langle x,y_0,z\rangle$
has defining relations:
$$
    x^{27}=1,\;
		x^{y_0}=x^{-8}z^3,\;
    x^9y_0^9=1,\;
    z^9=1,\;
    x^z=x^{-8},\;
    z^{y_0}=y_0^{-9}z.
$$
We next construct a cyclic extension $\langle x,y_0,z_0\rangle$ 
of $\langle x,y_0,z\rangle$ of order $3^8$, where
$z_0^{3}=z$, by means of an automorphism $\Omega_2$ 
of $\langle x,y_0,z\rangle$ that fixes $z$ and such that
$\Omega_2^{3}$ is conjugation by~$z$. With this goal in mind,
we consider the assignment
$$
x\mapsto x^{\alpha},\; y_0\mapsto y_0^\beta,\; z\mapsto z,
$$
where $\beta\in\Z$ satisfies $\alpha\beta\equiv 1\mod 27$. 
It is easy to see that all defining relations of $\langle x,y_0,z\rangle$ are preserved. Thus
the above assignment extends to an endomorphism $\Omega_2$ of $\langle x,y_0,z\rangle$, which is clearly surjective
and hence an automorphism. We next verify that $\Omega_2^{3}$ agrees with conjugation by $z$ on $x$, $y_0$, and $z$.
This is obviously true for $z$, and noting that $\alpha^3\equiv -8\mod 27$ and $\beta^3\equiv 10\mod 27$,
we find that it is also true for $x$ and $y$.
This produces the required extension, 
where $\Omega_2$ is conjugation by $z_0$. We readily verify that $\langle x,y_0,z_0\rangle$
has defining relations:
$$
    x^{27}=1,\;
		x^{y_0}=x^{-8}z_0^9,\;
    x^9y_0^9=1,\;
    z_0^{27}=1,\;
    x^{z_0}=x^{\alpha},\;
    {}^{z_0}y_0 = y_0^\alpha.
$$
We next construct a cyclic extension $\langle x_0,y_0,z_0\rangle$ 
of $\langle x,y_0,z_0\rangle$ of order $3^9$, where
$x_0^{3}=x$, by means of an automorphism $\Omega_3$ 
of $\langle x,y_0,z_0\rangle$ that fixes $x$ and such that
$\Omega_3^{3}$ is conjugation by~$x$. For this purpose,
we consider the assignment
$$
x\mapsto x,\; y_0\mapsto z_0^{-3}y_0^{-2},\; z_0\mapsto z_0 x^{-k}.
$$
The first, fourth and fifth defining relations of $\langle x,y_0,z_0\rangle$ are easily seen
to be preserved. As for the third, as $z$ commutes with $y$, it suffices to verify that $(z^{-1} y_0)^9=y_0^9$.
As indicated earlier, $y_0^z=y_0^{10}$, so indeed $(z^{-1} y_0)^9=y_0^{10+10^2+\cdots+10^9}z^{-9}=y_0^9$,
since $10(10^9-1)/9\equiv 9\mod 27$.

Regarding the second relation, we need to verify that $x^{z_0^{-3}y_0^{-2}}=x^{-8}(z_0 x^{-k})^9$. 
As $z$ normalizes~$\langle x\rangle$, $y$ commutes with $x$, and $z^3$ commutes with $x$, we have
$$
x^{z_0^{-3}y_0^{-2}}=x^{z^{-1} y_0}=(x^{10})^{y_0}=(x^{y_0})^{10}=(x^{-8}z^3)^{10}=xz^3.
$$
On the other hand, 
$$
(z_0 x^{-k})^9=z_0^9 x^{-k(1+\alpha+\cdots+\alpha^8)}=z^3 x^{-9k}=x^{-9k}z^3,
$$
since $(\alpha^9-1)/(\alpha-1)\equiv 9\mod 27$. Thus
$$
x^{-8}(z_0 x^{-k})^9=x^{-8}x^{-9k}z^3=x x^{-9-9k}z^3=xz^3,
$$
as $9(k+1)\equiv 0\mod 27$. This proves that the second relation is preserved.

In regards to the sixth relation, we need to verify that ${}^{z_0 x^{-k}}(z_0^{-3}y_0^{-2}) = (z_0^{-3}y_0^{-2})^\alpha$.
We have
$$
(z_0^{-3}y_0^{-2})^\alpha=(z^{-1}y_0y_0^{-3})^\alpha=(z^{-1}y_0)^\alpha y^{-\alpha}.
$$
Here
$$
(z^{-1}y_0)^\alpha=y_0^{10(1+10+\cdots+10^{\alpha-1})}z^{-\alpha}=y_0^{10\alpha}z^{-\alpha},
$$
since $(10^\alpha-1)/9\equiv \alpha\mod 27$. Since $k\equiv -1\mod 3$, we infer
$$
(z_0^{-3}y_0^{-2})^\alpha=y_0^{10\alpha}z^{-\alpha}y^{-\alpha}=y_0^{7\alpha}z^{-\alpha}=y_0^{7\alpha}z^{2}.
$$
On the other hand, from $x^{y_0}=x^{-8}z^3$ we deduce $(y_0^{-1})^x=x^{-9}z^3y_0^{-1}$. Here $y_0,x^3,z^3$
commute with each other, so $y_0^x=y_0x^{9}z^{-3}$. As $x$ commutes with $z^3$ and 
$k\equiv -1\mod 3$, it follows that $^{x^{-k}} y_0=y_0^{x^k}=y_0x^{-9}z^{3}$. In addition, from $x^z=x^{-8}$,
we deduce ${}^{x^{-1}} z^{-1}=x^{-9}z^{-1}=z^{-1}x^{-9}$, so ${}^{x^{-k}} z^{-1}=z^{-1}x^{-9k}=z^{-1}x^9$. Therefore,
$^{x^{-k}} (z_0^{-3}y_0^{-2})=z^{-1} x^9 y_0^{-2} x^{-9}z^{3}=y_0^{7}z^{2}$, whence 
${}^{z_0 x^{-k}} (z_0^{-3}y_0^{-2})=y_0^{7\alpha}z^2$. This demonstrates that the sixth relation is preserved.
Thus the above assignment extends to an endomorphism $\Omega_3$ of $\langle x,y_0,z_0\rangle$, which is clearly surjective
and hence an automorphism. We next verify that $\Omega_3^{3}$ agrees with conjugation by $x$ on $x$, $y_0$, and $z_0$.
This is obvious for $x$. Moreover, $\Omega_3^{3}$ sends $z_0$ to $z_0x^{-3k}=z_0x^{1-\alpha}$, and the given relation
$x^{z_0}=x^\alpha$ is equivalent to $z_0^x=z_0 x^{1-\alpha}$. We claim that $\Omega_3^{3}$ sends $y_0$ to
$y_0^x=x^9 z^{-3}y_0 $. Indeed, the hypothesis $k\equiv -1\mod 3$ implies $1+\alpha+\alpha^2\equiv 3\mod 27$, so
$$
z^{\Omega_3}=(z_0^3)^{\Omega_3}=(z_0 x^{-k})^3=z_0^3 x^{-k(1+\alpha+\alpha^2)}=z_0^3 x^{-3k}=z x^{-3k},
$$
and therefore $(z^{-1})^{\Omega_3}=z^{-1}x^{3k}$. By definition, we also have $y_0^{\Omega_3}=z^{-1}y_0^{-2}$, so
$$
y_0^{\Omega_3^2}=z^{-1}x^{3k}(z^{-1}y_0^{-2})^{-2}=z^{-1}x^{3k}(y_0^{2}z)^2=x^{3k}(y_0^2)^z y_0^2 z=x^{3k}y_0^{-5}z,
$$
$$
y_0^{\Omega_3^3}=x^{3k}(z^{-1}y_0^{-2})^{-5}zx^{-3k}=x^{3k}(y_0^{2}z)^{5}zx^{-3k}=x^{3k}z^5 y_0^{2\times 10(1+10+\cdots+10^4)}zx^{-3k}=
x^{3k}z^5 y_0^{10}zx^{-3k},
$$
since $(10^5-1)/9\equiv 14\mod 27$, so $2\times 10\times 14\equiv 10\mod 27$. As $x^9y_0^9=1$ and $[x^3,z]=1=[x^3,y_0]$, we infer
$$
y_0^{\Omega_3^3}=z^5 y_0^{10}z=z^5 y_0^9 y_0 z=x^{-9}z^6 y_0^z=x^{-9}z^6 y_0^{10}=x^9 z^{-3}y_0,
$$
as required. This produces the required extension, 
where $\Omega_3$ is conjugation by $x_0$. We see that $\langle x_0,y_0,z_0\rangle$
has defining relations:
$$
    x_0^{81}=1,\;
		y_0^{x_0}=z_0^{-3}y_0^{-2},\;
    x_0^{27}y_0^9=1,\;
    z_0^{27}=1,\;
    x_0^{z_0}=x_0^{\alpha},\;
    {}^{z_0}y_0 = y_0^\alpha.
$$
Observe that $x_0^{z_0}=x_0^{\alpha}$ is equivalent to $z_0^{x_0}=z_0x_0^{1-\alpha}=z_0 x_0^{-3k}=z_0 x^{-k}$, and
that $y_0^{x_0}=z_0^{-3}y_0^{-2}$ implies $(x_0^3)^{y_0}=x^{y_0}=x^{-8}z_0^9=x_0^{-24}z_0^9$. Indeed, we have just seen that
$x_0^{-3}y_0 x_0^3=y_0^{x_0^3}=y_0^x=y_0 z_0^{-9} x^9=y_0 z_0^{-9} x_0^{27}$, so $(x_0^{-3})^{y_0}=z_0^{-9}x_0^{24}$.

We finally construct a cyclic extension $\langle x_0,y_1,z_0\rangle$ 
of $\langle x_0,y_0,z_0\rangle$ of order $3^{10}$, where
$y_1^{3}=y_0$, by means of an automorphism $\Omega_4$ 
of $\langle x_0,y_0,z_0\rangle$ that fixes $y_0$ and such that
$\Omega_4^{3}$ is conjugation by~$y_0$. With  this goal in mind,
we consider the assignment
$$
x_0\mapsto x_0 z_0,\; y_0\mapsto y_0,\; z_0\mapsto y_0^{k}z_0.
$$
We proceed to show that the defining  relations of $\langle x_0,y_0,z_0\rangle$ are preserved.

We have $(x_0z_0)^{81}=z_0^{81}x_0^{\alpha(\alpha^{81}-1)/(\alpha-1)}=1$, because $(\alpha^{81}-1)/(\alpha-1)\equiv 0\mod 81$.

We next claim that $y_0^{x_0z_0}=(y_0^k z_0)^{-3}y_0^{-2}$. Indeed, on the one hand, we have 
$y_0^{x_0z_0}=(z_0^{-3}y_0^{-2})^{z_0}=z_0^{-3}y_0^{-2\beta}$, and on the other hand, since $1+\alpha+\alpha^2\equiv 3\mod 27$,
we have $(y_0^k z_0)^{3}=y_0^{k(1+\alpha+\alpha^2)}z_0^3=y_0^{3k}z_0^3$, so $(y_0^k z_0)^{-3}y_0^{-2}=z_0^{-3}y_0^{-3k-2}$,
and we are left to show that $2\beta\equiv 3k+2\mod 27$ or, equivalently, $(3k+2)(3k+1)\equiv 2\mod 27$,
which is true because $k\equiv -1\mod 3$.

We next prove that $(x_0z_0)^{27}y_0^9=1$. This is equivalent to  $(x_0z_0)^{27}=x_0^{27}$, which is true, since
$(x_0z_0)^{27}=z_0^{27}x_0^{\alpha(\alpha^{27}-1)/(\alpha-1)}$, $z_0^{27}=1$, and $(\alpha^{27}-1)/(\alpha-1)\equiv 27\mod 81$.

We next show that $(y_0^k z_0)^{27}=1$. We already computed $(y_0^k z_0)^{3}=y_0^{3k}z_0^3$, where $y_0^3=y$ and $z_0^3=z$
commute, so $(y_0^k z_0)^{27}=(y_0^{3k}z_0^3)^9=y_0^{27k}z_0^{27}=1$.

Next, from ${}^{z_0}y_0 = y_0^\alpha$, we derive ${}^{y_0^kz_0}y_0 = y_0^\alpha$, so the sixth relation is also preserved.

As for the fifth and last relation, we need to see that $(x_0z_0)^{y_0^k z_0}=(x_0z_0)^{\alpha}$. We have,
$$
(x_0z_0)^{\alpha}=z_0^\alpha x_0^{\alpha(1+\alpha+\cdots+\alpha^{\alpha-1})}.
$$
To find a formula for $x_0^{y_0}$, we start with $x_0^{-1}y_0x_0=y_0^{x_0}=z_0^{-3}y_0^{-2}$, and derive
$$
(x_0^{-1})^{y_0}x_0=y_0^{-1}x_0^{-1}y_0x_0=y_0^{-1}z_0^{-3}y_0^{-2}=y_0^{-1}z_0^{-3}y_0y_0^{-3}=(z^{-1})^{y_0}y^{-1}.
$$
Here we use the given relation $z^{y_0}=y_0^{-9}z=y^{-3}z$ to deduce $(z^{-1})^{y_0}=z^{-1}y^3$, where $y$ and $z$ commute.
Thus
$$
(x_0^{-1})^{y_0}=(z^{-1})^{y_0}y^{-1}x_0^{-1}=z^{-1}y^2x_0^{-1},
$$
$$
x_0^{y_0}=x_0y^{-2}z.
$$
Repeatedly using $x_0^{y_0}=x_0y^{-2}z$, $z^{y_0}=y^{-3}z$, $y^9=1$, and $[y,z]=1$, we find that
$$
x_0^{y_0}=x_0y^{7}z,\; x_0^{y_0^2}=x_0y^{2}z^2,\; x_0^{y_0^3}=x_0y^{3}z^3,\; x_0^{y_0^4}=x_0y z^4,\;
x_0^{y_0^5}=x_0y^5 z^5,
$$
$$
x_0^{y_0^6}=x_0y^{6}z^6,\; x_0^{y_0^7}=x_0y^{4}z^7,\; x_0^{y_0^8}=x_0y^{8}z^8,\; x_0^{y_0^9}=x_0,
$$
where the last relation is compatible with $x_0^{27}y_0^9=1$. Since $k\in\N$ and $k\equiv -1\mod  3$, we deduce
$$
x_0^{y_0^k}=x_0 y^k z^k.
$$

On the other hand, from $z_0 y_0 z_0^{-1}={}^{z_0} y_0=y_0^\alpha$, we derive  
$y_0^{-1}z_0 y_0 z_0^{-1}=y_0^{\alpha-1}$, so $z_0^{y_0}=y_0^{\alpha-1}z_0$, and therefore
$$
z_0^{y_0^k}=y_0^{(\alpha-1)k}z_0=y_0^{3k^2}z_0=y^{k^2}z_0.
$$
It follows that
$$
(x_0z_0)^{y_0^k}=x_0 y^{k} z^k y^{k^2}z_0=x_0 y^{k(k+1)} z_0^\alpha,
$$
$$
(x_0z_0)^{y_0^k z_0}=(x_0 y^{k(k+1)} z_0^\alpha)^{z_0}=x_0^\alpha y^{k(k+1)\beta} z_0^\alpha.
$$
Here $k+1\equiv 0\mod 3$, $\beta\equiv 1\mod 3$, and $y^9=1$, so
$$
(x_0z_0)^{y_0^k z_0}=x_0^\alpha y^{k(k+1)} z_0^\alpha.
$$
Now $k+1=3u$ with $u\in\N$. Making use of relations $x^9y^3=1$, $x_0^3=x$, and $y^9=1$, we see that
$$
y^{k(k+1)}=y^{(-1+3u)3u}=y^{-3u}=x^{9u}=x_0^{27u},
$$
$$
(x_0z_0)^{y_0^k z_0}=x_0^\alpha x_0^{27 u} z_0^\alpha.
$$
Thus, we are left to show that
\begin{equation}
\label{prr}
z_0^\alpha x_0^{\alpha(1+\alpha+\cdots+\alpha^{\alpha-1})}=x_0^\alpha x_0^{27 u} z_0^\alpha.
\end{equation}
From $x_0^{z_0}=x_0^\alpha$, $\alpha\equiv 1\mod 3$, and $x_0^{81}=1$, we see that $[x_0^{27},z_0]=1$. Thus (\ref{prr})
is equivalent to
$$
x_0^{\alpha(1+\alpha+\cdots+\alpha^{\alpha-1})}=x_0^{\alpha \alpha^\alpha} x_0^{27 u},
$$
that is,
$$
x_0^{-27 u}=x_0^{\alpha\gamma}.
$$
Assume first $v_3(k+1)\geq 2$. Then $3|u$, so $x_0^{-27 u}=1$. Moreover, $v_3(\gamma)\geq 4$ by Proposition~\ref{calc1},
so $x_0^{\alpha\gamma}=1$ as well. Suppose next that $v_3(k+1)=1$. Then $v_3(\gamma)=3$ with
$\gamma=27t$, $t\in\N$, and $t\equiv -u\mod 3$ by Proposition \ref{calc1}, so 
$\alpha\gamma\equiv (1+3k)27t\equiv 27t\equiv -27u\mod 81$, and $x_0^{\alpha\gamma}=x_0^{-27u}$. 
This proves that the fifth defining relation of $\langle x_0,y_0,z_0\rangle$ is preserved.
Thus the above assignment extends to an endomorphism $\Omega_4$ of $\langle x_0,y_0,z_0\rangle$, which is clearly surjective
and hence an automorphism. We next verify that $\Omega_4^{3}$ agrees with conjugation by $y_0$ on $x_0$, $y_0$, and $z_0$.
This is obvious for $y_0$. Moreover, from $z_0^{\Omega_4}=y_0^k z_0$, we infer $z_0^{\Omega_4^3}=
y_0^{3k} z_0=y_0^{\alpha-1} z_0$, which is indeed equal to $z_0^{y_0}$ as indicated earlier. 
Applying the definition of $\Omega_4$ twice yields 
$$x_0^{\Omega_4^2}=x_0z_0y_0^k z_0=x_0y_0^{k\alpha}z_0^2,$$
and a third application gives
$$x_0^{\Omega_4^3}=x_0z_0y_0^{k\alpha}(y_0^{k}z_0)^2=x_0y_0^{k(\alpha+2\alpha^2)} z.$$
On the other hand, we computed earlier $x_0^{y_0}=x_0 y^{-2}z$, and we are reduced to demonstrate that
$k(\alpha+2\alpha^2)\equiv -6\mod 27$. Making the substitution $\alpha=1+3k$, with $k=-1+3u$,
we see that this congruence holds. This produces the required extension, 
where $\Omega_4$ is conjugation by $y_1$. Thus $x_0^{y_1}=x_0 z_0$, so $z_0=[x_0,y_1]$,
which implies that $\langle x_0,y_1,z_0\rangle=\langle x_0,y_1\rangle$. Moreover,
we have $x_0^{z_0}=x_0^{z_0}$ and $y_1^{-1}z_0 y_1=z_0^{y_1}=y_0^k z_0=y_1^{3k}z_0$, so
$z_0y_1z_0^{-1}=y_1^{1+3k}$, that is, ${}^{z_0} y_1=y_1^\alpha$. Finally, we also have
$x_0^{81}=1=y_1^{81}$, with $|\langle x_0,y_1\rangle|=3^{10}$, so the proof is complete.
\end{proof}

\noindent{\bf Acknowledgments.} We thank Volker Gebhardt for his help with Magma computations,
and Andrea Previtali and the referee for a careful reading of the paper and useful comments.



\begin{thebibliography}{RBMW}

\small

\bibitem[A]{A} M. A. Albar \emph{On Mennicke groups of deficiency zero I},
Internati. J. Math. $\&$ Math. Sci. 8 (1985) 821--824.

\bibitem[AA]{AA} M. A. Albar and A.-A. A. Al-Shuaibi \emph{On Mennicke groups of deficiency zero II},
Can. Math. Bull. 34 (1991) 289--293.





\bibitem[AT]{AT} S. Aivazidis and T. M\"{u}ller 
\emph{Finite non-cyclic $p$-groups whose number of subgroups is minimal},
Arch. Math. 114 (2020) 13--17.


\bibitem[B]{B} Y. Berkovich \emph{Groups of prime power order, vol. 1}, Walter de Gruyter, 2008.



\bibitem[BC]{BC} J.N.S Bidwell and M.J. Curran \emph{The automorphism group of a split metacyclic $p$-group},
Arch. Math. 87 (2006) 488--497.


\bibitem[BJ]{BJ} Y. Berkovich and Z. Janko \emph{Groups of prime power order, vol. 2}, Walter de Gruyter, 2008. 

\bibitem[BJ2]{BJ2} Y. Berkovich and Z. Janko \emph{Groups of prime power order, vol. 3}, Walter de Gruyter, 2011. 

\bibitem[BJ3]{BJ3} Y. Berkovich and Z. Janko \emph{On subgroups of finite p-groups}, Israel J. Math. 171 (2009) 29--49. 

\bibitem[BJ4]{BJ4} Y. Berkovich and Z. Janko \emph{Structure of finite $p$-groups with given subgroups}, Contemp. Math. 402 (2006) 
13--93.

 \bibitem[C]{C} M.J. Curran \emph{The automorphism group of a split metacyclic $2$-group},
Arch. Math. 89 (2007) 10--23.

\bibitem[C2]{C2} M.J. Curran \emph{The automorphism group of a nonsplit metacyclic $p$-group},
Arch. Math. 90 (2008) 483--489.








\bibitem[D]{D} R.M. Davitt \emph{The automorphism group of a finite metacyclic $p$-group},
Proc. Amer. Math. Soc. 25 (1970) 876--879.

\bibitem[Di]{Di} J. Dietz \emph{Automorphisms of $p$-groups given as cyclic-by-elementary Abelian central extensions},
J. Algebra 242 (2001) 417--432.

\bibitem[G]{G} G. Glauberman \emph{Centrally large subgroups of finite $p$-groups}, 
J. Algebra 300 (2006) 480--508.

\bibitem[GG]{GG} M. Golasi\'nski and D. Gon\c{c}alves \emph{On automorphisms of split metacyclic groups},
Manuscripta Math. 128 (2009) 251--273.

\bibitem[H]{H} G. Higman \emph{A finitely generated infinite simple group},
J. Lond. Math. Soc. 26 (1951) 61--64.


\bibitem[IY]{IY} I.M. Isaacs and L. Yanovski \emph{Counting maximal abelian subgroups of $p$-groups},
Arch. Math. 119 (2022) 1--9.

\bibitem[J]{J} D.L. Johnson \emph{Presentations of groups. 2nd ed.},
London Mathematical Society Student Texts. 15. Cambridge University Press, Cambridge, 1997.


\bibitem[Ja]{Ja} E. Jabara \emph{Gruppi fattorizzati da sottogruppi ciclici},
Rend. Semin. Mat. Univ. Padova 122 (2009) 65--84.






\bibitem[JR]{JR} D.L. Johnson and E. F. Robertson \emph{Finite groups of deficiency zero},
in Homological Group Theory (ed. C.T.C. Wall), Cambridge University Press, 1979.




\bibitem[LM]{LM} C.R. Leedham-Green and S. McKay \emph{The structure of groups of prime power order},
London Mathematical Society Monographs, 27. Oxford, Oxford University Press, 2002.

 \bibitem[M]{M} I.D Macdonald \emph{On a class of finitely presented groups},
Canad. J. Math 14 (1962) 602--613.

\bibitem[M2]{M2} I.D Macdonald \emph{A computer application to finite $p$-groups},
J. Aust. Math. Soc. 17 (1974) 102--112.

\bibitem[Ma]{Ma} I. Malinowska \emph{The automorphism group of a split metacyclic $2$-group and some groups of crossed homomorphisms},
Arch. Math. 93 (2009) 99--109.

\bibitem[Ma2]{Ma2} I. Malinowska \emph{On the structure of the automorphism group of a minimal nonabelian $p$-group (metacyclic case)},
Glas. Mat. 47(67) (2012) 153--164.



\bibitem[Me]{Me} J. Mennicke \emph{Einige endliche Gruppen mit drei Erzeugenden und drei Relationen}, 
Arch. Math. (Basel) 10 (1959) 409--418.

\bibitem[Men]{Men} F. Menegazzo \emph{Automorphisms of $p$-groups with cyclic commutator subgroup},
Rend. Sem. Univ. Padova 90 (1993) 81--101.


\bibitem[MS]{MS} A. Montoya Ocampo and F. Szechtman \emph{The automorphism group of finite $p$-groups associated to the Macdonald group}, arXiv:2303.06385.

\bibitem[MS2]{MS2} A. Montoya Ocampo and F. Szechtman \emph{The automorphism group of finite $2$-groups associated to the Macdonald group}, 
arXiv:2308.03510.


\bibitem[MS3]{MS3} A. Montoya Ocampo and F. Szechtman \emph{On the isomorphism problem for certain $p$-groups}, 
arXiv:2308.03513.



\bibitem[PS]{PS} A. Previtali and F. Szechtman \emph{A study of the Wamsley groups and its Sylow subgroups},
preprint.






\bibitem[Ro]{Ro} D. J. S. Robinson \emph{A course in the theory of groups},
Graduate Texts in Mathematics, 80. Springer- Verlag, New York, 1980.

\bibitem[S]{S} E. Schenkman \emph{A factorization theorem for groups and Lie algebras},
Proc. Am. Math. Soc. 68 (1978) 149--152.




\bibitem[TW]{TW} G. Traustason and J. Williams \emph{Powerfully nilpotent groups},
J. Algebra 522 (2019) 80--100.

\bibitem[W]{W} J.W. Wamsley \emph{The deficiency of finite groups},
Ph.D. thesis, Univ. of Queensland (1969).


\bibitem[W2]{W2} J.W. Wamsley \emph{A class of three-generator, three-relation, finite groups},
Canad. J. Math. 22 (1070) 36--40.



\bibitem[W3]{W3} J.W. Wamsley \emph{A class of finite groups with zero deficiency}, 
Proc. Edinb. Math. Soc., II. Ser. 19 (1974) 25--29.

\bibitem[Z]{Z} H. J. Zassenhaus \emph{The theory of groups},
Dover, New York, 1999.



\end{thebibliography}
\end{document}